\documentclass[11pt,reqno,a4paper]{amsart}
\usepackage[british]{babel}

\usepackage{comment}
\usepackage{multicol}
\usepackage{graphicx}
\usepackage{amscd}
\usepackage{amsmath}
\usepackage{amsfonts}
\usepackage{amssymb}
\usepackage{comment}
\usepackage{color}
\usepackage{pdfsync}
\usepackage{bbm}
\usepackage{hyperref}
\usepackage{enumerate}
\usepackage{mathabx}
\usepackage{xcolor}
\usepackage[misc,geometry]{ifsym}

\definecolor{trp}{rgb}{1,1,1}

\definecolor{red}{rgb}{1,0,.2}

\newtheorem{theorem}{Theorem}[section]
\theoremstyle{plain}

\newtheorem{claim}{Claim}
\newtheorem*{claim*}{Claim}

\newtheorem*{example*}{Example}

\newtheorem{lemma}[theorem]{Lemma}

\newtheorem{prop}[theorem]{Proposition}

\newtheorem*{remark*}{Remark}

\numberwithin{equation}{section}

\newtheorem*{theorem*}{Theorem}

\newtheorem{theorem MTP}{Mass Transference Principle}
\newtheorem*{theorem MTP*}{Mass Transference Principle}
\newtheorem{theorem K}{Khintchine's Theorem}
\newtheorem*{theorem K*}{Khintchine's Theorem}
\newtheorem{theorem J}{Jarn\'{\i}k's Theorem}
\newtheorem*{theorem J*}{Jarn\'{\i}k's Theorem}

\newcommand{\R}{\mathbb{R}}
\newcommand{\N}{\mathbb{N}}

\newcommand{\ov}{\boldsymbol\omega}
\newcommand{\tv}{\boldsymbol\tau}

\newcommand{\ii}{\mathbf{i}}

\newcommand{\jj}{\mathbf{j}}
\newcommand{\x}{\mathbf{x}}
\newcommand{\A}{{\underline{A}}}

\newcommand{\PP}{\mathbb{P}}
\newcommand{\Sm}{\mathbb{S}}

\newcommand{\y}{\underline{\mathbf{y}}}

\newcommand{\diam}{\mathrm{diam}}

\renewcommand{\Bbb}[1]{\mathbb{#1}}
\newcommand{\Q}{{\Bbb Q}}         

\newcommand{\Z}{{\Bbb Z}}         

\newcommand{\cA}{{\mathcal A}}

\newcommand{\cH}{{\mathcal H}}

\newcommand{\bad}{\mathbf{Bad}}
\newcommand{\QI}{\mathbf{QI}}

\DeclareMathOperator{\dimh}{dim_H}

\begin{document}
\title[On the Hausdorff measure of shrinking target sets on self-conformal sets]{On the Hausdorff measure of shrinking target sets on self-conformal sets}

\author{Demi Allen}
\address[Demi Allen]{School of Mathematics, University of Bristol, Fry Building, Woodland Road, Bristol, BS8 1UG, UK, and the Heilbronn Institute for Mathematical Research, Bristol, UK}
\email{demi.allen@bristol.ac.uk}

\author{Bal\'azs B\'ar\'any}
\address[Bal\'azs B\'ar\'any]{Budapest University of Technology and Economics, Department of Stochastics, MTA-BME Stochastics Research Group, P.O.Box 91, 1521 Budapest, Hungary}
\email{balubsheep@gmail.com}

\subjclass[2010]{Primary: 28A80, 37C45 Secondary: 11J83}
\keywords{Self-conformal sets, shrinking targets, Hausdorff measure, Diophantine approximation.}
\thanks{The first author was supported in part by EPSRC Doctoral Prize Fellowship grant EP/N509565/1 and the Heilbronn Institute for Mathematical Research. The second author acknowledges support from the grants OTKA K123782, NKFI PD123970, and the J\'anos Bolyai Research Scholarship of the Hungarian Academy of Sciences.}

\begin{abstract}
In this article, we study the Hausdorff measure of shrinking target sets on self-conformal sets. The Hausdorff dimension of the sets we are interested in here was established by Hill and Velani in 1995. However, until recently, little more was known about the Hausdorff measure of these particular sets. In this paper we provide a complete characterisation of the Hausdorff measure of these sets, obtaining a dichotomy for the Hausdorff measure which is determined by the convergence or divergence of a sum depending on the radii of our ``shrinking targets''. Our main result complements earlier work of Levesley, Salp, and Velani~(2007), and recent work of Baker (2019). 
\end{abstract}
\date{\today}

\maketitle

\thispagestyle{empty}

\section{Introduction}

The central theme of Diophantine approximation is understanding how well real numbers can be approximated by rationals and various generalisations thereof. Given an approximating function $\psi: \N \to \R^+:=[0,\infty)$, the classical set of \emph{$\psi$-well-approximable points} is defined as the points $x\in[0,1]$ such that $\left|x-p/q\right|<\psi(q)$ for infinitely many $(p,q) \in \Z \times \N.$ In \emph{metric} Diophantine approximation, there is a particular emphasis on understanding the ``size'' of such sets, in particular, the Lebesgue and Hausdorff measures. The most classical results in this direction can be traced back to Khintchine \cite{kh,Khintchine ref} and Jarn\'{i}k \cite{Jarnik31} from the 1920s and 1930s who proved beautiful zero-full dichotomies for, respectively, the Lebesgue and Hausdorff measures of the $\psi$-well-approximable points in $\R^t$. Throughout we will denote the Hausdorff $s$-measure of $F$ by $\cH^s(F)$, and we will denote the Hausdorff dimension of $F$ by $\dimh{F}$. For the definitions and properties of Hausdorff measure and dimension we refer the reader to~\cite{falconerfractalgeom}.

In 1995, Hill and Velani \cite{HillVelanifirst} introduced a natural analogue of the classical $\psi$-well-approximable points from Diophantine approximation to the theory of dynamical systems, these are the so-called \emph{shrinking target sets}. The basic idea is that we have a point, say $y \in X$ where $X$ is a metric space, and a decreasing sequence $(r_n)_{n \in \N}$ of positive real numbers. The balls $ B(y,r_n)$ are our \emph{targets}. Given a map $T: X \to X$, we are interested in the set of points $x \in X$ for which $T^nx$ hits the target $B(y,r_n)$ infinitely often. Since their initial introduction, many more authors have contributed to the study of \emph{shrinking targets}, \emph{moving targets}, and \emph{recurrence sets}. To name but a few, see \cite{baker2019quantitative, HillVelanifirst, HillVelanitori, Jarvenpaa_et_al_2017, Jarvenpaa2017, PerssonRams} and references within.

In this paper, we are particularly interested in the shrinking target problem on self-conformal sets. Fix $t \in \N$ and let $D\subset\R^t$ be a simply connected compact set such that $D=\overline{D^o}$. Here we denote the interior of a set $F \subset \R^t$ by $F^o$ and we denote its closure by~$\overline{F}$. We say that a map $f\colon D\mapsto D$ is contracting if there exists a constant $0<c<1$ such that
\[\|f(x)-f(y)\|\leq c\|x-y\|\]
for every $x,y\in D$. Throughout, $\|\cdot\|:\R^t \to \R^+$ will be any fixed norm on $\R^t$, and $B(x,r)$ denotes an open ball in $\R^t$, with respect to the fixed norm $\|\cdot\|$, centered at $x$ with radius~$r$.

Let $\Lambda$ be a finite set of symbols and let $\Phi=\{f_i\colon D\mapsto D\}_{i\in\Lambda}$ be a finite set of contracting mappings. We call $\Phi$ an \emph{iterated function system} (IFS).  Hutchinson \cite{Hutchinson} showed that there exists a unique non-empty compact set $X$ such that
$$
X=\bigcup_{i\in\Lambda}f_i(X).
$$
We call $X$ the \emph{attractor} of $\Phi$. Throughout the paper we will assume that the maps $\{f_i\}_{i \in \Lambda}$ are $C^{1+\varepsilon}$-conformal regular strict contractions, i.e. the maps are $C^{1+\varepsilon}$-conformal and there exist constants $0<a_{\min}\leq a_{\max}<1$ such that
\begin{align}\label{f' bounds}
a_{\min}\leq\min_{x\in D}\|f_i'(x)\|\leq\max_{x\in D}\|f_i'(x)\|\leq a_{\max}
\end{align}
for every $i\in\Lambda$, where $\|f_i'(x)\|$ denotes the operator norm of the linear map $y \to f_i'(x)y$. We say that a mapping $f\colon D\mapsto D$ is $C^{1+\varepsilon}$-conformal if for every $x\in D$ and $y\in\R^t$,
\[\|f'(x)y\|=\|f'(x)\|\|y\|\text{ and }\|f'(x)-f'(y)\| \leq C\|x-y\|^{\varepsilon}\text{ for all $x,y \in D$}.\]
Note that the second condition above just says that the mapping $x \to f'(x)$ is H\"{o}lder continuous with exponent $\varepsilon>0$. If the iterated function system $\Phi$ consists of $C^{1+\varepsilon}$-conformal mappings, we refer to $\Phi$ as a \emph{conformal iterated function system} and we call the attractor of the conformal IFS $\Phi$ a \emph{self-conformal set}. 

In the special case when the IFS consists of similarities, i.e. there exist constants $0<a_i<1$ for each $i \in \Lambda$ such that for all $x,y \in D$ we have
\[\|f_i(x)-f_i(y)\|=a_i\|x-y\|,\]
we say that $\Phi$ is a \emph{self-similar iterated function system} and that the attractor is a \emph{self-similar set}. Clearly, all self-similar sets are self-conformal.

Let us denote by $\Lambda^*$ the set of all finite sequences formed of the symbols in $\Lambda$, i.e. \mbox{$\Lambda^*=\bigcup_{n=0}^{\infty}\Lambda^n.$} We will denote the length of $\ii\in\Lambda^*$ by $|\ii|$. We will also use the following shorthand notation:
\[f_\ii=f_{i_1}\circ f_{i_2}\circ\cdots\circ f_{i_{|\ii|}}, \quad \text{and} \quad X_{\ii}=f_{\ii}(X).\]
We will adopt the convention that $f_\emptyset(x)=x$. We observe that by combining \eqref{f' bounds} with the chain rule, we also have for any $\ii \in \Lambda^*$ that
\begin{align}\label{f' composition bounds}
a_{\min}^{|\ii|}\leq\min_{x\in D}\|f_{\ii}'(x)\|\leq\max_{x\in D}\|f_{\ii}'(x)\|\leq a_{\max}^{|\ii|}.
\end{align}

The dimension theory of shrinking target sets is well understood in the case of conformal iterated function systems which satisfy the open set condition. We say that $\Phi$ satisfies the \emph{open set condition} (OSC) if there exists an open set $U\subset D$ such that
\begin{equation}\label{eq:OSC}
f_i(U)\subset U\text{ for every } i \in \Lambda,\text{ and }f_i(U) \cap f_j(U)=\emptyset \text{ for all }i\neq j \in \Lambda.
\end{equation}

Given an approximating function $\Psi\colon\Lambda^*\mapsto\R^+$, for each $x\in X$ we will be interested in the set
\begin{equation}\label{eq:shrinking1}
W(x,\Psi):=\{y\in X:\|y-f_\ii(x)\|<\Psi(\ii)\text{ for infinitely many }\ii\in\Lambda^*\}.
\end{equation}
We note that although the set $W(x, \Psi)$ is defined using strict inequalities, the geometric measure properties (in particular, Hausdorff measure and dimension) of the set would be unchanged if we instead defined the set using non-strict inequalities.

The Hausdorff dimension of $W(x,\Psi)$ was established by Hill and Velani in~\cite{HillVelanifirst} for $\Psi: \Lambda^* \to \R^+$ of the form
\begin{align} \label{Psi special form}
\Psi(\ii)=\diam(X_{\ii})\psi(|\ii|)
\end{align}
where $\psi: \N \to \R^+$ is any monotonic decreasing approximating function (see Theorem \ref{thm:hillvelani} for a precise statement). Note that under the OSC, the maps $f_i$ can be interpreted as the local inverses of some expanding map $T$ on $D$, and so $W(x,\Psi)$ with $\Psi$ of the form \eqref{Psi special form} is a proper generalisation of the shrinking target sets of expanding dynamics to iterated function systems.

In recent years, several other authors have also studied the problem of shrinking targets on fractals. For example, Chernov and Kleinbock \cite{CherKlein} studied the measure of shrinking target sets with respect to ergodic measures, Chang, Wu and Wu \cite{ChangWuWu2019} very recently studied the problem of recurrence sets on linear iterated iterated function systems consisting of maps with equal contraction ratios, Koivusalo and Ram\'{\i}rez \cite{KoivusaloRamirez} considered shrinking targets on self-affine sets, the second author and Rams computed the Hausdorff dimension for certain shrinking targets on Bedford-McMullen carpets \cite{BaranyRams2018}, and Seuret and Wang considered some related problems in the setting of conformal iterated function systems~\cite{SeuretWang}. However, the value of the $\dimh(W(x,\Psi))$-dimensional Hausdorff measure remained unknown except in some very special cases. The main result of the present paper is the following theorem which provides a complete characterisation of the Hausdorff $s$-measure of the set $W(x,\Psi)$ when $\Psi$ takes the form given in \eqref{Psi special form}.

\begin{theorem}[Main result]\label{thm:main} Let $\Phi$ be a conformal iterated function system which satisfies the open set condition and has attractor $X$. Let $x\in X$ and let $\psi\colon\N\mapsto\R^+$ be a monotonic decreasing function. Let $\Psi(\ii)=\mathrm{diam}(X_\ii)\psi(|\ii|)$ and let $W(x,\Psi)$ be as in \eqref{eq:shrinking1}. Then, for any open ball $B\subset\R^t$,
	$$\mathcal{H}^s(W(x,\Psi)\cap B)=\begin{cases}
	0 & \text{if }\sum_{\ii\in\Lambda^*}\Psi(\ii)^s<\infty,\\&\\
	\mathcal{H}^s(X\cap B) & \text{if }\sum_{\ii\in\Lambda^*}\Psi(\ii)^s=\infty.
	\end{cases}
	$$
\end{theorem}

One of the previous examples where such complete characterisation was known is the work of Levesley, Salp and Velani \cite{LevSalpVel}, which was the original motivation of our paper. In~\cite{LevSalpVel}, Levesley, Salp and Velani considered points in general \emph{missing digit sets} which are $\psi$-well-approximable by rationals with appropriate denominators. A classical example of a missing digit set is the middle-third Cantor set  and the work of Levesley, Salp, and Velani is just one of many works which aims to address the following question of Mahler posed in 1984 in his famous note ``Some Suggestions for Further Research'' \cite{Mahler84}:
\begin{quotation}
	\emph{How close can irrational elements of Cantor's set be approximated by rational numbers \\
		(i) in Cantor's set, and \\
		(ii) by rational numbers not in Cantor's set?} \\
\end{quotation}  Suppose $b \geq 3$ is an integer and $J(b)$ is any proper subset of the set $S(b):=\{0,1,\dots,b-1\}$ with $\#J(b) \geq 2$. Denote by $K_{J(b)}$ the missing digit set consisting of the numbers $x \in [0,1]$ for which there exists a base-$b$ expansion consisting only of digits from $J(b)$. For example, the middle-third Cantor set is an example of a missing digit set corresponding to taking $b=3$ and $J(b)=\{0,2\}$. It can easily be verified that the Hausdorff dimension of $K_{J(b)}$ is
$\gamma^* := \dimh{K_{J(b)}} = \frac{\log{\#J(b)}}{\log{b}}.$

Now, let $\cA(b) = \{b^n: n=0,1,2,\dots\}$ and let $\phi: \N \to \R^+$ be an approximating function. Then let
\[W_{\cA(b)}(\phi):=\left\{x \in [0,1]: \left|x-\frac{p}{q}\right|<\phi(q) \text{ for infinitely many } (p,q) \in \Z \times \cA(b) \right\}.\]
It is easy to see that $W_{\cA(b)}(\phi)\cap K_{J(b)}$ corresponds to the shrinking target set $W(0,\Psi) \cup W(1,\Psi)$ with $\psi(n)=b^n\phi(b^n)$. In \cite{LevSalpVel}, Levesley, Salp and Velani state the following result.\footnote{In \cite[Theorem 4]{LevSalpVel} the result is stated for general gauge functions $f$ whereas here for simplicity we opt to state the result only in terms of Hausdorff $s$-measure.} 

\begin{theorem} [Levesley -- Salp -- Velani, \cite{LevSalpVel}] \label{LSV theorem}
	Let $s \geq 0$. Then,
	\[
	\cH^{s}(W_{\cA(b)}(\phi) \cap K_{J(b)})=
	\begin{cases}
	0&\text{if}\quad\sum_{n=1}^{\infty}{\phi(b^n)^s \times (b^n)^{\gamma^*}}<\infty,\\
	&\\
	\cH^s (K_{J(b)})&\text{if}\quad\sum_{n=1}^{\infty}{\phi(b^n)^s \times (b^n)^{\gamma^*}}=\infty.
	\end{cases}
	\]
\end{theorem}

The Mass Transference Principle due to Beresnevich and Velani \cite{bervel} plays a crucial role in \cite{LevSalpVel} in the proof of Theorem \ref{LSV theorem} to show the divergence part. However, interestingly, the Mass Transference Principle cannot be applied in the generality of our setting. More precisely, even in the self-similar case when the contraction ratios are inhomogeneous (that is there exist maps $f_i,f_j$ with strictly different contraction ratios in absolute value) the conditions of the Mass Transference Principle are violated, see Section~\ref{sec:nomasstrans}. That said, although it is impossible for us to apply the Mass Transference Principle in our general setting, it is worth remarking that the overarching strategy we use for proving Theorem~\ref{thm:main} is still inspired by some of the underlying ideas in the proof of the Mass Transference Principle in \cite{bervel}.

More recently, complementary to the Hausdorff dimension result of Hill and Velani \cite{HillVelanifirst}, Baker has made the first steps towards understanding the Hausdorff measure of the set $W(x,\Psi)$ on self-conformal sets. 

\begin{theorem}[Baker, \cite{baker}]\label{thm:baker}
	Let $\Phi$ be a conformal iterated function system which satisfies the open set condition and has attractor $X$. Given $x \in X$ and $\Psi: \Lambda^* \to \R^+$, let $W(x,\Psi)$ be defined as in \eqref{eq:shrinking1}.
	\begin{enumerate}[(i)]
	\item{Let $s>0$ and suppose that
		\[\sum_{\ii \in \Lambda^*}{\Psi(\ii)^s} < \infty.\]
		Then, $\cH^s(W(x,\Psi))=0$ for all $x \in X$.}
	\item{Let $\psi: \N \to \R^+$ be a monotonic decreasing approximating function and suppose that $\Psi(\ii)=\diam(X_{\ii})\psi(|\ii|)$. Then, if
		\[\sum_{\ii \in \Lambda^*}{\Psi(\ii)^{\dimh(X)}} = \infty\]
		we have
		\[\cH^{s}(W(x,\Psi^{\frac{\dimh{X}}{s}})) = \cH^{s}(X)\]
		for all $x \in X$ and $s \leq \dimh{X}$. Here, $\Psi^t(\ii)=\Psi(\ii)^t=(\diam(X_{\ii})\psi(|\ii|))^t$.}
\end{enumerate}
\end{theorem}

We note that Theorem \ref{thm:baker} $(i)$ holds for any iterated function systems, not just those which are conformal and satisfy the open set condition. More recently, Baker has extended Theorem \ref{thm:baker} to also deal with overlapping iterated function systems in \cite{baker2019overlapping}. To prove Theorem \ref{thm:baker} $(ii)$ Baker establishes a Jarn\'{\i}k-type zero-full dichotomy for the Hausdorff $\dimh{X}$-measure and, in the case $s<\dimh{X}$, Baker applies the Mass Transference Principle. It is important to note here the main difference between our main result (Theorem \ref{thm:main}) and Baker's result. In particular, in order to study the Hausdorff-$s$ measure of the shrinking target set in the divergence case, Baker needs to modify the radii of the targets, that is, he considers $\Psi(\ii)^t=(\diam(X_{\ii})\psi(|\ii|))^t$, while we can preserve them in the original form. Hence, our result does not follow directly from Baker's. In particular, we will show that in the case of self-similar iterated function systems with non-uniform contraction ratios that for any \mbox{$\Psi(\ii)=\diam(X_{\ii})\psi(|\ii|)$} with $\sum_{\ii\in\Lambda^*}\Psi(\ii)^s=\infty$ there is no $\Psi'$ of the form $\Psi'(\ii)=\diam(X_{\ii})\psi'(|\ii|)$ such that $\sum_{\ii\in\Lambda^*}\Psi'(\ii)^{\dimh{X}}=\infty$ and $W(x,(\Psi')^{\frac{\dimh{X}}{s}})\subseteq W(x,\Psi)$, see Section~\ref{sec:nomasstrans}.

\subsection{Structure of the paper} \label{outline}

In the next section, we introduce some further preliminaries and in Section \ref{sec:nomasstrans} we provide a more detailed discussion of why we are unable to use the Mass Transference Principle directly in the present setting, and why our result (Theorem~\ref{thm:main}) does not follow from the result due to Baker (Theorem~\ref{thm:baker}).

We present the proof of Theorem \ref{thm:main} in Sections \ref{start of proof of main theorem}--\ref{end of proof section}. The key ingredient of the proof of Theorem \ref{thm:main} is Proposition~\ref{prop:main}. In Section \ref{start of proof of main theorem}, we provide a proof of Theorem \ref{thm:main} assuming Proposition \ref{prop:main}. In Sections \ref{symbolic section}--\ref{end of proof section} we establish Proposition \ref{prop:main}.

As hinted at previously, the ideas underlying our argument share some similarities with those present in the proof of the Mass Transference Principle \cite{bervel}. Namely, Proposition~\ref{prop:main} relies on the careful construction of a suitable Cantor-type set and a measure supported on this set satisfying certain conditions which enable us to use a version of the mass distribution principle. The existence of such a Cantor set together with an appropriate measure supported on this set is guaranteed by Proposition \ref{prop:massdist}. In fact, establishing Proposition \ref{prop:massdist} is arguably the most substantial part of the proof of Theorem \ref{thm:main}.

In Section \ref{symbolic section}, we describe the set $W(x,\Psi)$ in the language of symbolic dynamics and provide a proof of Proposition \ref{prop:main} subject to Proposition \ref{prop:massdist}. In Section \ref{mass distribution section} we describe the required mass distribution, proving several technical lemmas along the way, before completing the proof of Proposition \ref{prop:massdist} in Section \ref{end of proof section} by showing that the mass distribution we have constructed satisfies the requirements of Proposition \ref{prop:massdist}. This completes the proof of Theorem \ref{thm:main}.

Finally, in Section~\ref{Bad and QI application section} we give an application of our main theorem to approximating badly approximable numbers by quadratic irrationals.

\section{Preliminaries and Notations} \label{preliminaries}

\subsection{Conformal iterated function systems}

Throughout the paper, we will always assume that the iterated function system $\Phi$ consists of $C^{1+\varepsilon}$-conformal mappings, which map the simply connected compact set $D\subset\R^t$ with $D=\overline{D^o}$ into itself. If $f\colon D\mapsto D$ is a conformal map then by the simply connectedness of $D$, for every $x,y\in D$ there exist $\xi,\xi'\in D$ such that
\begin{equation} \label{MVT}
\|f'(\xi)\|\|x-y\|\leq\|f(x)-f(y)\|\leq\|f'(\xi')\|\|x-y\|.
\end{equation}
Notice that \eqref{MVT} is just a statement of the mean value theorem.

Since we insist that it consists of $C^{1+\varepsilon}$-conformal mappings, we have that the IFS $\Phi$ satisfies the \emph{bounded distortion property}. That is, there exists a constant $C \geq 1$ such that for every $\ii\in\Lambda^*$ and for all $x,y\in X$ we have
\begin{equation}\label{eq:boundeddist}
C^{-1}\leq\dfrac{\|f_{\ii}'(x)\|}{\|f_\ii'(y)\|}\leq C.
\end{equation}
For a proof of this fact we refer the reader to work of Simon, Solomyak and Urba\'nski \cite[Lemma~5.8]{SSU1}.

Combining \eqref{MVT} with the bounded distortion property we see that, for every $Y\subseteq D$, $\xi \in X$ and $\ii \in \Lambda^*$,
\begin{align} \label{diameter comparisons}
C^{-1}\diam(f_{\ii}(Y)) \leq \|f_{\ii}'(\xi)\|\diam(Y) \leq C\diam(f_{\ii}(Y)),
\end{align}
where $C$ is the constant appearing in \eqref{eq:boundeddist}.

The \emph{pressure function} $P\colon\R\mapsto\R$ corresponding to the IFS $\Phi$ is defined as follows,
\begin{equation}\label{eq:defP}
P(s)=\lim_{n\to\infty}\frac{1}{n}\log\sum_{\ii\in\Lambda^n}\|f_{\ii}'(x)\|^s.
\end{equation}
Note that by the bounded distortion property \eqref{eq:boundeddist}, the function $P$ is independent of the choice of $x\in X$. It is easy to see that $P$ is strictly monotonically decreasing, convex and continuous.

Peres, Rams, Simon and Solomyak \cite[Theorem~1.1]{PRSS} showed that $\Phi$ satisfies the open set condition if and only if
\begin{equation}\label{eq:finitemeas}
0<\mathcal{H}^d(X)<\infty,
\end{equation}
where $d$ is the unique solution of the equation $P(d)=0$. In this case $d=\dimh X$. This result was later generalized by K\"aenm\"aki and Rossi \cite[Proposition~3.5]{KaenRoss}.

Let $\psi\colon\N\mapsto\R^+$ be a monotonically decreasing function and let us define the \emph{shrinking rate} of the function $\psi$ as
\[\alpha=\alpha(\psi)=\liminf_{n\to\infty}-\frac{\log\psi(n)}{n}.\]
Note that $\alpha$ can be $+\infty$. The following theorem can be deduced from \cite[Theorem~7]{HillVelanifirst}.

\begin{theorem}[Hill -- Velani, \cite{HillVelanifirst}]\label{thm:hillvelani} Let $\Phi$ be a conformal iterated function system which satisfies the open set condition and has attractor $X$. Let $x\in X$ and let $\psi\colon\N\mapsto\R^+$ be a monotonic decreasing function. Denote by $\alpha:=\alpha(\psi)$ the shrinking rate of $\psi$. Then, for the approximating function $\Psi(\ii)=\mathrm{diam}(X_\ii)\psi(|\ii|)$ and the set $W(x,\Psi)$, we have
$$
\dimh (W(x,\Psi))=\begin{cases}
0 & \text{if }\alpha=\infty, \\
s & \text{if }0\leq\alpha<\infty,
\end{cases}
$$
where $s$ is the unique root of the equation $P(s)=s\alpha$. 
\end{theorem}

Note that the equation $P(s)=s\alpha$ has always a unique solution $s\in[0, d]$ for every $\alpha\geq0$, since the map $s\mapsto P(s)-s\alpha$ is strictly monotonically decreasing, $P(0)=\log(\#\Lambda)$ and $P(d)-d\alpha=-d\alpha\leq0$.

\subsection{Symbolic approach}

Throughout, let $\Sigma=\Lambda^\N$ and let $\sigma\colon\Sigma\mapsto\Sigma$ denote the usual left-shift operator on $\Sigma$; namely, for $\ii = (i_1,i_2,i_3,i_4,\dots) \in \Sigma$,
\[\sigma\ii = \sigma(i_1,i_2,i_3,i_4,\dots) = (i_2,i_3,i_4,\dots).\]
For $\ii=(i_1,\ldots,i_k)\in\Lambda^*$, the cylinder set $[\ii]$ is defined as
$$
[\ii]=[i_1,\ldots,i_k]:=\{\jj\in\Sigma:i_1=j_1,\ldots,i_k=j_k\}.
$$
By convention, $[\emptyset]=\Sigma$. For a sequence $\ii\in\Sigma$ and $n, m \in \N$ with $n\leq m$, let
\[\ii|_n^m:=(i_n,\ldots,i_m).\]
For $n>m$ we define $\ii|_n^m=\emptyset$. Next, for $\ii,\jj\in\Sigma$ let $|\ii\wedge\jj|=\min\{k\geq1:i_k\neq j_k\}-1$ and let $\ii\wedge\jj=i_1,\ldots,i_{|\ii\wedge\jj|}$ be the common part of $\ii$ and $\jj$. If $|\ii\wedge\jj|=0$ then we define $\ii\wedge\jj$ as the empty word. 

For any $\alpha\in[0,\infty)$, there exists a unique ergodic $\sigma$-invariant probability measure, $\mathbb{P}$, and a constant $C \geq 1$ such that for every $\ii\in\Lambda^*$ and $x\in X$
\begin{equation}\label{eq:measP}
C^{-1}\leq\dfrac{\mathbb{P}([\ii])}{e^{-\alpha s|\ii|}\|f_{\ii}'(x)\|^s}\leq C,
\end{equation}
where $s$ is the unique solution of the equation $P(s)=s\alpha$, see for example \cite[Theorem~1.2]{Bowen}.

The elements of $\Sigma$ and $X$ can be associated in a natural way. More precisely, for every $\ii\in\Sigma$ let
$$
\pi(\ii)=\lim_{n\to\infty}f_{i_1}\circ\cdots\circ f_{i_n}(0).
$$
We call the function $\pi\colon\Sigma\mapsto X$ the \emph{natural projection}. It is easy to see that $\pi(\ii)=f_{i_1}(\pi(\sigma\ii))$. In particular, for any $n \in \N$, we have
\begin{align} \label{pi composition equality}
\pi(\ii)=f_{\ii|_1^n}(\pi(\sigma^n\ii)).
\end{align}

For a probability measure $\mu$ on $\Sigma$, we will denote by $\pi_*\mu$ the \emph{pushforward} of the measure $\mu$; that is, for a set $A \subseteq X$ we have $\pi_*\mu(A)=\mu(\pi^{-1}(A))$.

\section{A Limitation of the Mass Transference Principle}\label{sec:nomasstrans}

Before we prove the main result of this paper, we show two phenomena. First, we show that the Mass Transference Principle, originally introduced by Beresnevich and Velani in~\cite{bervel}, is not applicable if the underlying iterated function system consists of similarities with different contraction ratios, by showing that the assumptions of the Mass Transference Principle are violated in this case. Secondly, we show that Theorem \ref{thm:main} cannot be deduced from Baker's result Theorem~\ref{thm:baker}~\cite{baker}.

More precisely, in this section, we consider iterated function systems of the form
\begin{equation}\label{eq:IFS}
\Phi=\{f_i\colon x\in\R^t\mapsto a_iO_ix+b_i\}_{i\in\Lambda},
\end{equation}
where $a_i \in (0,1)$, $b_i \in \R^t$, and $O_i$ is a rotation. We will be particularly interested in the case when there exist $i, j \in \Lambda$ such that $a_i \neq a_j$. We note that for iterated function systems of the form in \eqref{eq:IFS}, the pressure equation from \eqref{eq:defP} simplifies to 
\begin{align}\label{self-similar pressure}
P(s)=\log\left(\sum_{i\in\Lambda}a_i^s\right).
\end{align} 
In this case, we also note that the value of $s$ appearing in the Hausdorff dimension result due to Hill and Velani (Theorem \ref{thm:hillvelani}) is the unique solution of $\sum_{i\in\Lambda}a_i^se^{-s\alpha}=1$.

Before proceeding, we state a corresponding version of the Mass Transference Principle which is the most relevant to our current setting. The statement we give below, Theorem~\ref{mtp}, can be deduced from \cite[Theorem 3]{bervel}, which was the result used by Levesley, Salp, and Velani in \cite{LevSalpVel} to study Diophantine approximation on the middle-third Cantor set. Indeed, since its initial discovery, the Mass Transference Principle has become a widely used tool with profound consequences Diophantine approximation.


\begin{theorem}[Beresnevich -- Velani, \cite{bervel}] \label{mtp}
Let $\Phi$ be a conformal iterated function system of the form \eqref{eq:IFS} which satisfies the open set condition and has attractor $X$. Let us write $d=\dimh X$ and let $s\in(0,d]$ be arbitrary. Let $\psi\colon\N\mapsto\R^+$ be an arbitrary bounded map. For $\Psi(\ii)=\mathrm{diam}(X_\ii)\psi(|\ii|)$ and $x\in X$, denote by $W(x,\Psi)$ the set defined in \eqref{eq:shrinking1}. If, for any ball $B\subset\R^t$,
	$$
	\mathcal{H}^d(B\cap W(x,\Psi^{s/d}))=\mathcal{H}^d(B\cap X)
	$$
	then, for any ball $B\subset\R^t$, $$\mathcal{H}^s(B\cap W(x,\Psi))=\mathcal{H}^s(B\cap X).$$
\end{theorem}

It follows from the result of Hill and Velani (Theorem~\ref{thm:hillvelani}) that $\dimh W(x,\Psi)$ is the unique solution of the equation $P(s)=\alpha(\psi) s$, where $\alpha(\psi)$ is the shrinking rate of $\psi$. Hence, the case when the value of the Hausdorff measure is in question is exactly at this choice of $s$. We show that if $0<\alpha(\psi)<\infty$ and there exist $i,j \in \Lambda$ such that $a_i \neq a_j$ then the $d$-dimensional Hausdorff measure of the limsup set $W(x,\Psi^{s/d})$ would be zero at this critical choice of $s$. In such cases the Theorem~\ref{mtp} is not useful.

For simplicity, we will work with iterated function systems $\Phi = \{f_i\}_{i \in \Lambda}$ satisfying the \emph{strong separation condition} (SSC). That is, we assume that
\[
f_i(X)\cap f_j(X)=\emptyset\text{ for every }i\neq j \in \Lambda,
\]
where $X$ is the attractor of $\Phi$. Clearly any IFS satisfying the SSC also satisfies the OSC.

\begin{prop}\label{prop:nogood1}
	Let $\Phi$ be a conformal iterated function system of the form \eqref{eq:IFS} which satisfies the strong separation condition and has attractor $X$. Suppose that there exist $i,j\in\Lambda$ such that $a_i\neq a_j$. Let $\psi\colon\N\mapsto\R^+$ be a monotonic decreasing function with shrinking rate $\alpha(\psi)\in(0,\infty)$. Let $s$ be the solution to $\sum_{i\in\Lambda}e^{-\alpha s}a_i^s=1$ and write $d=\dimh X$. In this case it is well-known that $d$ is the unique solution of $\sum_{i\in\Lambda}a_i^d=1$.
	
	Let $x\in X$ and let $\Psi(\ii)=\mathrm{diam}(X_\ii)\psi(|\ii|)$. Then,
	$$
	\mathcal{H}^d(W(x,\Psi^{s/d}))=0.
	$$
\end{prop}

\begin{proof}
	Without loss of generality, we may assume that $\diam(X)=1$. Throughout the proof, let $\x$ be the unique coding of $x$, i.e. $\pi(\x)=x$. The uniqueness of this encoding is guaranteed by the strong separation condition.
	
	Let us define $\lambda$ to be the natural $\sigma$-invariant ergodic probability measure on $\Sigma$, whose projection is equivalent to $\mathcal{H}^d|_X$. That is, $\pi_*\lambda=\lambda\circ\pi^{-1}=\frac{\mathcal{H}^d|_X}{\mathcal{H}^d(X)}$. Moreover, for each $\ii=(i_1,\ldots,i_n)\in\Lambda^*$,
	$$
	\lambda([i_1,\ldots,i_n])=(a_{\ii})^d.
	$$
	Here we use the notation $a_{\ii}=a_{i_1}a_{i_2}\dots a_{i_{|\ii|}}$.
	
	Let $\chi:=-\sum_{i\in\Lambda}a_i^d\log a_i=-\int\log a_{i_1}d\lambda(\ii)$ and let
	$$
	F=\left\{\ii\in\Sigma:\lim_{n\to\infty}\frac{-1}{n}\log a_{\ii|_1^n}=\lim_{n\to\infty}\frac{-1}{n}\sum_{k=1}^n\log a_{i_k}=\chi\right\}.
	$$
	By Birkhoff's Ergodic Theorem (see, for example, \cite[Theorem 2.30]{EinsiedlerWard} or \cite[Theorem~1.14]{Walters}), $\lambda(F)=1$.
	
	Now, let
	$$
	\widehat{W}(x,\Psi^{s/d}):=\left\{\jj\in\Sigma:\|\pi(\jj)-\pi(\ii\x)\|< (a_\ii)^{s/d}\psi(|\ii|)^{s/d}\text{ for infinitely many }\ii\in\Lambda^*\right\}.
	$$
	Since $\Phi$ satisfies the strong separation condition and hence points in the symbolic space $\Sigma$ uniquely encode points in $X$, we have that
	\begin{equation}\label{eq:proj}\pi \widehat{W}(x,\Psi^{s/d}) = W(x,\Psi^{s/d}).\end{equation}
	Thus, since $\lambda(F)=1$, it is sufficient for us to show that $\lambda(\widehat{W}(x,\Psi^{s/d})\cap F)=0$. We make the following claim.
	
	\begin{claim}\label{c:1} If $\jj\in\widehat{W}(x,\Psi^{s/d})\cap F$ then  $|\jj\wedge\ii\x|\geq|\ii|$ for infinitely many $\ii\in\Lambda^*$ such that $\|\pi(\jj)-\pi(\ii\x)\|< (a_\ii)^{s/d}\psi(|\ii|)^{s/d}$. Recall that $\jj\wedge\ii\x$ is the common part of $\jj$ with the concatenation $\ii\x$.
	\end{claim}
	
	\begin{proof}[Proof of Claim \ref{c:1}] Suppose to the contrary that there exists $N=N(\jj)\geq1$ such that if $|\ii|>N$ and $\|\pi(\jj)-\pi(\ii\x)\|\leq (a_\ii)^{s/d}\psi(|\ii|)^{s/d}$ then $|\jj\wedge\ii\x|<|\ii|$.
		
		Let 
\[\delta=\min_{i\neq j}d(f_i(X),f_j(X))>0\] 
where, for subsets $A,B\subset \R^t$, $d(A,B)=\min\{\|a-b\|:a \in A \text{ and } b \in B\}$. 
Then from the definitions of the common part $\ii \wedge \jj$ and $\delta>0$ and \eqref{pi composition equality} we have
		\begin{align*}
		\delta a_{\jj\wedge\ii\x}&\leq a_{\jj\wedge\ii\x}\|\pi(\sigma^{|\jj\wedge\ii\x|}\jj)-\pi(\sigma^{|\jj\wedge\ii\x|}\ii\x)\| \\
		&=\|f_{\jj\wedge\ii\x}(\pi(\sigma^{|\jj\wedge\ii\x|}\jj))-f_{\jj\wedge\ii\x}(\pi(\sigma^{|\jj\wedge\ii\x|}\ii\x))\|\\
		&=\|\pi(\jj)-\pi(\ii\x)\| \\
		&\leq (a_\ii)^{s/d}\psi(|\ii|)^{s/d}.
		\end{align*}
		The last inequality above holds by assumption.
		
		Since $s<d$ and we are assuming that $|\jj \wedge \ii\x| < |\ii|$, it follows from the previous inequality that
		$$
		\delta\leq (a_{\jj\wedge\ii\x})^{s/d-1}(a_{\sigma^{|\jj\wedge\ii\x|}\ii})^{s/d}\psi(|\ii|)^{s/d}\leq(a_{\jj\wedge\ii\x})^{s/d-1}\psi(|\ii|)^{s/d}\leq (a_{\jj|_1^{|\ii|}})^{s/d-1}\psi(|\ii|)^{s/d}.
		$$
		Thus, since $\jj\in F$, we have
		\begin{equation}\label{eq:contbaound}
		0\geq\liminf_{n\to\infty}\frac{-1}{n}\log\left((a_{\jj|_1^{n}})^{s/d-1}\psi(n)^{s/d}\right)=\left(\frac{s}{d}-1\right)\chi+\frac{s}{d}\alpha.
		\end{equation}
		
		Let
		\[D_{KL}(\lambda\|\PP):=\sum_{i\in\Lambda}\lambda(i)\log\frac{\lambda(i)}{\PP(i)}\] denote the \emph{Kullback-Leibler divergence} (or \emph{relative entropy}) of the measure $\lambda$ with respect to $\PP$, where $\PP$ is the measure defined in \eqref{eq:measP}. See \cite[Section 2.6]{MacKay} for a definition of Kullback-Leibler divergence. It is a property of the Kullback-Leibler divergence that $0\leq D_{KL}(\lambda\|\PP)$ and $0=D_{KL}(\lambda\|\PP)$ if and only if $\lambda=\PP$.
		
		Since $\alpha>0$, we have that $s<d$. By the assumption that there exist contraction ratios $a_i\neq a_j$, it follows that there exists $i\in\Lambda$ such that $e^{\alpha s}a_i^s\neq a_i^d$. Indeed, otherwise we would have that $a_i^{d-s}=e^{\alpha s}=a_j^{d-s}$ for every $i,j\in\Lambda$, which is impossible. Thus, $\lambda\neq\PP$ and it follows from the definitions of the measures $\PP$ and $\lambda$ that
		\begin{align} \label{KL equality}
		0<D_{KL}(\lambda\|\PP)=\left(s-d\right)\chi+s\alpha,
		\end{align}
		but this contradicts \eqref{eq:contbaound}.
	\end{proof}
	
	Claim~\ref{c:1} combined with \eqref{pi composition equality} implies that for every $\jj\in \widehat{W}(x,\Psi^{s/d})\cap F$ there are infinitely many $\ii\in\Lambda^*$ such that $\jj|_1^{|\ii|}=\ii$ and
	$$
	(a_{\ii})^{s/d}\psi(|\ii|)^{s/d}>\|\pi(\jj)-\pi(\ii\x)\|=\|f_{\ii}(\pi(\sigma^{|\ii|}\jj))-f_\ii(\pi(\x))\|=a_{\ii}\|\pi(\sigma^{|\ii|}\jj)-\pi(\x)\|.
	$$
	Hence,
	\begin{equation}
	\widehat{W}(x,\Psi^{s/d})\cap F \subseteq \{\ii\in\Sigma:\|\pi(\sigma^n\ii)-x\|<(a_{\ii|_1^n})^{s/d-1}\psi(n)^{s/d}\text{ for infinitely many }n \in \N\}\cap F.
	\end{equation}
	
	By the definition of $F$, it follows from Egorov's Theorem (see, for example, \cite[Theorem~12.1]{Khar}) that, for every $\varepsilon>0$ there exists a set $E\subset F$ such that
	$\lambda(E)>1-\varepsilon$ and the sequence of functions $\ii\mapsto\frac{-1}{n}\log a_{\ii|_1^n}$ converges uniformly to $\chi$ on $E$.
	
	In particular, there exists a natural number $N=N(E)$ such that for every $\ii\in E$ and every $n\geq N$
	$$
	a_{\ii|_{1}^{n}}^{s/d-1}\psi(n)^{s/d} < e^{-\frac{n}{2d}D_{KL}(\lambda\|\PP)}.
	$$
	To see this, recall the right-hand equalities of \eqref{eq:contbaound} and \eqref{KL equality}.
	Thus,
	\begin{align*}
	\widehat{W}(x,\Psi^{s/d})\cap E&\subseteq W'':=\{\jj \in \Sigma :\|\pi(\sigma^n\jj)-x\|< e^{-\frac{n}{2d}D_{KL}(\lambda\|\PP)}\text{ for infinitely many } n \in \N\}\\
	&=\{\jj \in \Sigma :\|\pi(\jj)-f_{\ii}(x)\|< a_{\ii} e^{-\frac{|\ii|}{2d}D_{KL}(\lambda\|\PP)}\text{ for infinitely many } \ii \in \Lambda^*\}.
	\end{align*}
	Since
	\[\sum_{\ii\in\Lambda^*}(a_{\ii} e^{-\frac{|\ii|}{2d}D_{KL}(\lambda\|\PP)})^d=\sum_{n=1}^\infty e^{-nD_{KL}(\lambda\|\PP)/2}\left(\sum_{i\in\Lambda}a_i^d\right)^n=\sum_{n=1}^\infty e^{-nD_{KL}(\lambda\|\PP)/2}<\infty,\]
	we have that $\lambda(W'')=\mathcal{H}^d(\pi(W''))=0$ by Theorem~\ref{thm:baker}. Hence, by \eqref{eq:proj} and the fact that $\lambda(F)=1$ we have
	\begin{align*}
	\mathcal{H}^d(W(x,\Psi^{s/d}))&=\mathcal{H}^d(X)\cdot\lambda(\widehat{W}(x,\Psi^{s/d}))\\
	&=\mathcal{H}^d(X)\cdot\lambda(\widehat{W}(x,\Psi^{s/d})\cap F)\\
	&\leq \mathcal{H}^d(X)\cdot\left(\lambda(W'')+\lambda(F\setminus E)\right) \\
	&<\mathcal{H}^d(X)\cdot\varepsilon.
	\end{align*}
	Finally, since $\varepsilon>0$ was arbitrary and $\cH^d(X)$ is finite by \eqref{eq:finitemeas}, the statement follows.
\end{proof}

{ Although Baker's result Theorem~\ref{thm:baker}(ii) relies on the Mass Transference Principle, {\it a priori} it might happen that for a given function $\psi$ and $s>0$ with $\sum_{\ii\in\Lambda^*}\Psi(\ii)^s=\infty$ one could construct another function $\overline{\psi}$ for which $\sum_{\ii\in\Lambda^*}\overline{\Psi}(\ii)^d=\infty$ implies $\mathcal{H}^s(W(x,\Psi))=\infty$ for any $x\in X$. One way to show this for example would be to show the containment $W(x,(\overline{\Psi})^{d/s})\subseteq W(x,\Psi)$ and apply Theorem~\ref{thm:baker}(ii). However, the next proposition shows that such a containment is not possible in some cases.}

\begin{prop}\label{prop:nogood2}
	Let $\Phi$ be a conformal iterated function system of the form \eqref{eq:IFS} which satisfies the strong separation condition and has attractor $X$. Write $d=\dimh X$. Suppose that there exist $i,j\in\Lambda$ such that $a_i\neq a_j$. Let $\psi\colon\N\mapsto\R^+$ be a monotonic decreasing function { such that the shrinking rate is achieved as a limit, i.e. 
	$$
0<\alpha(\psi)=\lim_{n\to\infty}\frac{-\log\psi(n)}{n}<d.
$$ Moreover, suppose that $\sum_{\ii\in\Lambda^*}\Psi(\ii)^s=\infty$, where $s$ is the unique solution of the equation $\sum_{i\in\Lambda}e^{-\alpha s}a_i^s=1$ and $\Psi(\ii)=\mathrm{diam}(X_\ii)\psi(|\ii|)$.} Then there exists $x\in X$ such that there is no function $\overline{\Psi}\colon\Lambda^*\mapsto\R^+$ of the form $\overline{\Psi}(\ii)=\mathrm{diam}(X_\ii)\overline{\psi}(|\ii|)$, where $\overline{\psi}\colon\N\mapsto\R^+$ is such that $\alpha(\overline{\psi})=\liminf_{n\to\infty}\frac{-\log\overline{\psi}(n)}{n}\geq0$ and $\sum_{\ii\in\Lambda^*}\overline{\Psi}(\ii)^d=\infty$ and $W(x,(\overline{\Psi})^{d/s})\subseteq W(x,\Psi)$.
\end{prop}

{ Observe that such a function $\psi$ exists, for example take $\psi(n)=e^{-\alpha n}$. }

\begin{proof}
	Without loss of generality, we may assume that $\diam(X)=1$, $a_1=\max_{i\in\Lambda}a_i$. Now, let $\x:=(2,2,2,\ldots)$ and $x:=\pi(\x)$. {In order to reach a contradiction, we suppose that there does exist} $\overline{\Psi}\colon\Lambda^*\mapsto\R^+$ of the form $\overline{\Psi}(\ii)=\mathrm{diam}(X_\ii)\overline{\psi}(|\ii|)$, where $\overline{\psi}\colon\N\mapsto\R_+$ is such that $\alpha(\overline{\psi})=\liminf_{n\to\infty}\frac{-\log\overline{\psi}(n)}{n}\geq0$ and $\sum_{\ii\in\Lambda^*}\overline{\Psi}(\ii)^d=\infty$ and $W(x,(\overline{\Psi})^{d/s})\subseteq W(x,\Psi)$.
	
	Since $\infty=\sum_{\ii\in\Lambda^*}\overline{\Psi}(\ii)^d=\sum_{n=0}^\infty\overline{\psi}(n)^d$ and $\alpha(\overline{\psi})\leq0$ we get $\alpha(\overline{\psi})=0$. Let $n_k$ be a sequence along which $\lim_{k\to\infty}\frac{-\log\overline{\psi}(n_{k})}{n_{k}}=0$. Let
	\begin{equation}\label{eq:elldef}
	\ell_k=\left\lceil\left(\frac{d}{s}-1\right)\left(\left(n_k-\sum_{i=1}^{k-1}\ell_i\right)\frac{\log a_1}{\log a_2}+\sum_{i=1}^{k-1}\ell_i\right)+\frac{d\log\overline{\psi}(n_{k})}{s\log a_2}\right\rceil.
	\end{equation}
	By taking a sufficiently fast growing subsequence of $n_k$ and replacing $n_k$ with it, we can assume without loss of generality that $n_k+\ell_k<n_{k+1}$ and $\frac{\sum_{i=1}^{k-1}\ell_i}{n_k}\to0$ as $k\to\infty$.
	
	Now, let us define $\ii\in\Sigma$ such that
	$$
	i_j=\begin{cases}
	2 & \text{ if }n_k\leq j\leq n_{k}+\ell_k\\
	1 & \text{ if }n_k+\ell_k<j< n_{k+1}
	\end{cases}\text{ for some $k\in\N$.}
	$$
	It is easy to see that $\pi(\ii)\in W(\pi(\x),(\overline{\Psi})^{d/s})$. Indeed, by \eqref{eq:elldef}
	$$
	\|\pi(\sigma^{n_k}\ii)-\pi(\x)\|\leq a_2^{\ell_k}\leq \left(a_{\ii|_{n_k}}\right)^{d/s-1}\overline{\psi}(n_k)^{d/s},
	$$
	and so $\|\pi(\ii)-f_{\ii|_{n_k}}(\pi(\x))\|\leq\left(\overline{\Psi}(\ii)\right)^{d/s}$. We will show that $\pi(\ii)\notin W(\pi(\x),\Psi)$. 
	
	By the disjointness of the cylinders, $\pi(\ii)\in W(\pi(\x),\Psi)$ if and only if 
	$$
	\|\pi(\sigma^{n_k}\ii)-\pi(\x)\|\leq\psi(n_k)\text{ for infinitely many $k {\in \N}$.}
	$$
	However, by the strong separation condition $\|\pi(\sigma^{n_k}\ii)-\pi(\x)\|\geq\delta a_2^{\ell_k-1}$, where $\delta=\min_{i\neq j}d(f_i(X),f_j(X))$. So $\psi(n_k)\geq\delta a_2^{\ell_k-1}$ for infinitely many $k{\in \N}$. Taking logarithms, dividing by $n_k$, and letting $k\to\infty$, we get
	$$
	-\alpha\geq \left(\frac{d}{s}-1\right)\log a_1.
	$$
	However, by the convexity of the pressure {\eqref{self-similar pressure}} we have
	$$
	P(s)\geq P(d)+(s-d)P'(d),
	$$
	where $P'(d)=\sum_{i\in\Lambda}a_i^d\log a_i$. Since there exists $a_i\neq a_1$ we have that $P'(d)<\log a_1$. {This implies}
	$$
	(d-s)(-\log a_1)\geq\alpha s=P(s)\geq (s-d)P'(d)>(d-s)(-\log a_1),
	$$
	which is a contradiction. {Hence,} $\pi(\ii)\notin W(\pi(\x),\Psi)$.
\end{proof}

Note that the Hausdorff measure of the sets $W(x,\Psi)$ in both Proposition~\ref{prop:nogood1} and Proposition~\ref{prop:nogood2} can be calculated by using Theorem~\ref{thm:main}.

\section{Proof of the Main Result (Theorem \ref{thm:main})} \label{start of proof of main theorem}

We first note that the convergence part of Theorem \ref{thm:main} is contained in greater generality in Theorem \ref{thm:baker} $(i)$. When $\sum_{\ii \in \Lambda^*}{\Psi(\ii)^s}<\infty$, the proof that $\cH^s(W(x,\Psi))=0$ follows from a standard covering argument combined with the definition of Hausdorff measure. For further details see the argument given in \cite[\S 3.2]{baker}. Thus, it remains to prove the divergence part of Theorem~\ref{thm:main}.

An observation that is central to proving the convergence part of the result, and which we will also make use of for the divergence case, is that the set $W(x,\Psi)$ is the $\limsup$ set of the family of balls $\{B(f_\ii(x),\Psi(\ii))\}_{\ii\in\Lambda^*}$. That is,
\[
W(x,\Psi)=\bigcap_{n=0}^\infty\bigcup_{\substack{ \ii\in\Lambda^*\\|\ii|\geq n}}B(f_\ii(x),\Psi(\ii)).\]

In proving Theorem~\ref{thm:main}, we first show that by using Theorems \ref{thm:hillvelani} and \ref{thm:baker} the problem can be reduced to the case when $P(s)=s\alpha$ where $\alpha:=\alpha(\psi)$ is the shrinking rate of $\psi$ and $P$ is the pressure function defined in \eqref{eq:defP}. To tackle the proof in the remaining case, we use the following proposition.

\begin{prop}\label{prop:main}
Let $\Phi$ be a conformal iterated function system which satisfies the open set condition and has attractor $X$. Let $x\in X$ and let $\psi\colon\N\mapsto\R^+$ be a monotonic decreasing function with shrinking rate $0<\alpha(\psi)<\infty$. Let $\Psi(\ii) = \diam(X_\ii)\psi(|\ii|)$ and let $s$ be the unique solution of the equation $P(s)=s\alpha(\psi)$. For the set $W(x,\Psi)$, if
\[\sum_{n=1}^\infty\psi(n)^s\sum_{\ii\in\Lambda^n}\|f_{\ii}'(x)\|^s=\infty,\]
then
\[\mathcal{H}^{s}(W(x,\Psi))=\infty.\]
\end{prop}

In the remainder of this section, we will give the proof of Theorem \ref{thm:main} assuming Proposition \ref{prop:main}. The rest of the paper will then be devoted to establishing Proposition \ref{prop:main} and a number of other required technical lemmas. The key to establishing Proposition \ref{prop:main} is proving the existence of a suitable mass distribution as outlined in Proposition \ref{prop:massdist}.

\begin{proof}[Proof of Theorem~\ref{thm:main} (Divergence)]
	Recall that we are given
\begin{align} \label{divergence sum}
\sum_{\ii\in\Lambda^*}\Psi(\ii)^s=\infty.
\end{align}
It follows from \eqref{diameter comparisons} together with the bounded distortion property \eqref{eq:boundeddist} that, for every $x\in X$, we have
\begin{align} \label{extra divergence}
\sum_{n=1}^\infty\psi(n)^s\sum_{\ii\in\Lambda^n}\|f_{\ii}'(x)\|^s=\infty.
\end{align}

Next, let $\alpha$ denote the shrinking rate of $\psi$. First of all, let us consider the case when $\alpha=\infty$. It follows from \eqref{f' composition bounds} that
$$
\infty=\sum_{n=1}^\infty\psi(n)^s\sum_{\ii\in\Lambda^n}\|f_{\ii}'(x)\|^s\leq\sum_{n=1}^\infty\psi(n)^s (\sharp\Lambda a_{\max}^s)^n,
$$
which is possible if and only if $s=0$. Otherwise, it follows from the definition of $\alpha$ that the terms in the sum on the far right-hand side become too small, thus forcing the sum to converge. However, it can be seen that $W(x,\Psi)$ has continuum many elements. Indeed, it is a countable intersection of open and dense sets and, hence, it is a dense $G_\delta$ set by Baire's category theorem, see \cite[Theorem~6.54]{HewStr}. Thus, $\mathcal{H}^0(W(x,\Psi)\cap B)=\infty$ for every open ball $B$ in $X$ and so we may assume that $\alpha<\infty$.
	
Next, observe that, by the root test, it follows from \eqref{extra divergence} that
\begin{align} \label{root test}
1\leq\limsup_{n\to\infty}\sqrt[n]{\psi(n)^s\sum_{\ii\in\Lambda^n}\|f_{\ii}'(x)\|^s}=e^{-\alpha s}e^{P(s)}.
\end{align}

	If $\alpha=0$ then $P(s) \geq 0$. Furthermore, in this case, $\dimh(W(x,\Psi))=d:=\dimh X$ by Theorem \ref{thm:hillvelani} and \eqref{eq:finitemeas}, which yields that $\dimh X = d$ where $d$ is the unique solution of the equation $P(d)=0$. Since $P(s)$ is strictly monotonically decreasing, if $P(s)>0$ then $s<d$ and, hence, $\mathcal{H}^s(W(x,\Psi))=\infty$. If $P(s)=0$ then $s=d=\dimh X$ and the statement follows from Theorem~\ref{thm:baker}$(ii)$ since $W(x,\Psi)$ has full $\cH^d|_X$-measure. Thus, we may assume that $0<\alpha<\infty$.
	
	Note that it follows from \eqref{root test} that $P(s) \geq s\alpha$. { Since $s\mapsto P(s)-s\alpha$ is strictly monotonically decreasing, } if $P(s)>s\alpha$ then by Theorem~\ref{thm:hillvelani}, $\dimh(W(x,\Psi))>s$ and thus the statement follows again. So, for the remainder of the proof, suppose that $P(s)=s\alpha$.
	
	Now, let $B$ be an open ball such that $X\cap B\neq\emptyset$. Since the maps of $\Phi$ are uniformly contracting {(this is essentially what \eqref{f' bounds} tells us)}, there exists $\ii\in\Lambda^*$ so that $f_\ii(X)\subseteq B$. Let $\widetilde{\Psi}(\jj)=\diam(X_\jj)\widetilde{\psi}(|\jj|)$, where $\widetilde{\psi}(n)=C^{-1}\psi(n+|\ii|)$ and $C>1$ is the constant in \eqref{eq:boundeddist}. From \eqref{extra divergence} and the bounds on $f'$ given in \eqref{f' composition bounds}, we have
	\[
	\begin{split}
	\sum_{n=0}^\infty\widetilde{\psi}(n)^s\sum_{\jj\in\Lambda^n}\|f_{\jj}'(x)\|^s&=C^{-s}\sum_{n=0}^\infty\psi(n+|\ii|)^s\sum_{\jj\in\Lambda^n}\|f_{\jj}'(x)\|^s\\
	&\geq C^{-s}(\sharp\Lambda)^{-|\ii|}a_{\max}^{-s|\ii|}\sum_{n=0}^\infty\psi(n+|\ii|)^s\sum_{\jj\in\Lambda^{n+|\ii|}}\|f_{\jj}'(x)\|^s=\infty.
	\end{split}
	\]
	
Thus, by Proposition~\ref{prop:main}, we have $\mathcal{H}^s(W(x,\widetilde{\Psi}))=\infty$.

Next, suppose that $\|y-f_\jj(x)\|<\widetilde{\Psi}(\jj)$. Employing the mean value theorem \eqref{MVT} and \eqref{diameter comparisons}, we see that
\begin{align*}	
\|f_{\ii}(y)-f_{\ii\jj}(x)\| &\leq \|f_{\ii}'(\xi)\|\|y-f_{\jj}(x)\| \\
                             &< \|f_{\ii}'(\xi)\|\widetilde{\Psi}(|\jj|) \\
                             &=C^{-1}\|f_{\ii}'(\xi)\|\diam(X_\jj)\psi(|\ii\jj|) \\
                             &\leq \diam(f_\ii(X_\jj))\psi(|\ii\jj|)\\
                             &= \diam(X_{\ii\jj})\psi(|\ii\jj|).
\end{align*}
	Hence, if $y\in W(x,\widetilde{\Psi})$ then $f_{\ii}(y)\in W(x,\Psi)$ and thus
	$$
	\mathcal{H}^s(W(x,\Psi)\cap B)\geq\mathcal{H}^s(f_{\ii}(W(x,\widetilde{\Psi}))\cap B)=\mathcal{H}^s(f_{\ii}(W(x,\widetilde{\Psi}))).
	$$
The last inequality above follows since $\ii \in \Lambda^*$ was chosen so that $f_{\ii}(X) \subset B$ and so, consequently, we also have $f_{\ii}(W(x,\widetilde{\Psi}) \subset B$.

Finally, it follows from \eqref{f' composition bounds} and \eqref{diameter comparisons} combined with the definition of Hausdorff \mbox{$s$-measure} that
\[\mathcal{H}^s(f_{\ii}(W(x,\widetilde{\Psi}))) \geq a_{\min}^{|\ii|s}\mathcal{H}^s(W(x,\widetilde{\Psi})).\]
This completes the proof.	
\end{proof}

\section{Proof of Proposition \ref{prop:main}: A Symbolic Approach} \label{symbolic section}

For the rest of the paper, we fix an $x\in X$ and a symbolic representation $\x \in \Sigma$ for which $\pi(\x)=x$. Next, let us define $\rho\colon\N\mapsto\N$ as follows; let $\rho(n)$ be the unique natural number such that
\begin{equation}\label{eq:defrho}
\diam(X_{\x|_1^{\rho(n)}})\leq C^{-2}\diam(X)\psi(n)<\diam(X_{\x|_1^{\rho(n)-1}}).
\end{equation}
Note that $\rho$ is monotonically increasing. By combining \eqref{eq:defrho} with \eqref{diameter comparisons} and the bounded distortion property \eqref{eq:boundeddist}, it can be seen that if $\alpha$ is the shrinking rate of $\psi$ defined earlier, then
\begin{align} \label{alpha rho}
	\alpha=\liminf_{n\to\infty}\dfrac{-\log\psi(n)}{n}=\liminf_{n\to\infty}\frac{-1}{n}\log\diam(X_{\x|_1^{\rho(n)}})=\liminf_{n\to\infty}\frac{-1}{n}\log\|f_{\x|_1^{\rho(n)}}'(x)\|.
\end{align}

For a monotonic increasing function $\rho\colon\N\mapsto\N$ let
\begin{equation}\label{eq:shrinking3}
\widehat{W}(\x,\rho)=\{\ii\in\Sigma:\sigma^n\ii\in[\x|_1^{\rho(n)}]\text{ for infinitely many }n\in\N\}.
\end{equation}

\begin{lemma}\label{lem:incl}
	Let $\x\in\Sigma$, let $\psi\colon\N\mapsto\R^+$ be a monotonic decreasing function, and let $\rho$ be as defined in \eqref{eq:defrho}. Then,
\[\pi\widehat{W}(\x,\rho)\subseteq W(\pi(\x),\Psi).\]
\end{lemma}

\begin{proof}
	If $\ii\in\widehat{W}(\x,\rho)$, then $\sigma^n\ii\in[\x|_1^{\rho(n)}]$ for infinitely many $n\in\N$. For each such $n \in \N$, we have $\pi(\sigma^n\ii)\in\pi([\x|_1^{\rho(n)}])=X_{\x|_1^{\rho(n)}}$. Hence, by the definition of $\rho$, for infinitely many $n\in\N$ we have
\[\|\pi(\sigma^n\ii)-\pi(\x)\|\leq\diam(X_{\x|_1^{\rho(n)}})\leq C^{-2}\psi(n)\mathrm{diam}(X).\]

Then, using \eqref{MVT}, \eqref{diameter comparisons}, \eqref{pi composition equality} and the fact that $\pi(\x)=x$, for infinitely many $n \in \N$ we have
\begin{align*}
\|\pi(\ii)-f_{\ii|_1^{n}}(x)\|&\leq C\|f_{\ii|_1^{n}}'(\xi)\|\|\pi(\sigma^n\ii)-x\| \\
&\leq C^{-1}\|f_{\ii|_1^{n}}'(\xi)\|\diam(X)\psi(n)\\
&\leq\mathrm{diam}(X_{\ii|_1^{n}})\psi(n) \\
&=\Psi(\ii|_1^{n}). 
\end{align*}
Thus, $\pi(\ii)\in W(x,\Psi)$.
\end{proof}

Our goal now is to prove the following proposition, which implies Proposition~\ref{prop:main}.

	For simplicity, throughout the rest of the paper, we use the Vinogradov notation and write $A\ll B$ to denote that $A\leq dB$ for some constant $d > 0$. When we refer to explicit constants $C$, these may not always be the same constant but will typically be related to the bounded distortion property \eqref{eq:boundeddist} or the constant arising in \eqref{eq:measP}.

\begin{prop}\label{prop:massdist}
Let $\x\in\Sigma$ and let $\rho\colon\N\mapsto\N$ be a monotonic increasing function with
\[\liminf_{n\to\infty}\frac{-1}{n}\log\|f_{\x|_1^{\rho(n)}}'(\pi(\x))\|=:\alpha\in(0,\infty).\]
Let $s$ be the unique solution of the equation $P(s)=s\alpha$. Suppose that
\[\sum_{n=0}^\infty e^{\alpha s n}\|f_{\x|_1^{\rho(n)}}'(\pi(\x))\|^{s}=\infty.\]
Then, there exists a probability measure $\eta$ such that $\eta(\widehat{W}(\x,\rho))=1$ and, for every $\delta>0$, there exists a $K\geq1$ such that for every $\ii\in\Lambda^*$ with $|\ii|\geq K$,
$$\eta([\ii])\ll\delta\cdot (\diam(X_\ii))^s,$$
where the implicit constant is independent of $\ii$ and $\delta$.
\end{prop}

For $r>0$, denote by $\Theta_r$ the sequences $\ii \in \Lambda^*$ for which the cylinders $f_\ii(X)$ have diameter approximately equal to $r$. More precisely,
$$
\Theta_r=\{\ii\in\Lambda^*:\diam(X_\ii)\leq r<\diam(X_{\ii|_1^{|\ii|-1}})\}.
$$
Note that the collection of cylinders $\{[\ii]:\ii\in\Theta_r\}$ partitions $\Sigma$.

\begin{proof}[Proof of Proposition~\ref{prop:main}]
	First let us show that the assumption on divergence implies the divergence of the series in Proposition~\ref{prop:massdist}. Take $\mathbb{P}$ to be the measure described in \eqref{eq:measP} and let $\rho: \N \to \N$ be as defined in \eqref{eq:defrho}. Then, from the definitions of $\PP$ and $\rho$ it follows by \eqref{eq:boundeddist} and \eqref{diameter comparisons} that
	\[
	\begin{split}
	\infty=\sum_{n=0}^\infty\psi(n)^s\sum_{\ii\in\Lambda^n}\|f_{\ii}'(x)\|^s&\leq C^{2s+1}\diam(X)^{-s}\sum_{n=0}^\infty\diam(X_{\x|_1^{\rho(n)-1}})^s\sum_{\ii\in\Lambda^n}e^{\alpha s n}\PP([\ii])\\
	&\leq C^{3s+1}\sum_{n=0}^\infty\|f_{\x|_1^{\rho(n)-1}}'(f_{x_{\rho(n)}}(x))\|^{s}e^{\alpha s n}\\
	&\leq C^{3s+1}\sum_{n=0}^\infty a_{\min}^{-s}\|f_{x_{\rho(n)}}'(x)\|^s\|f_{\x|_1^{\rho(n)-1}}'(f_{x_{\rho(n)}}(x))\|^{s}e^{\alpha s n}\\
	&= C^{3s+1}a_{\min}^{-s}\sum_{n=0}^\infty\|f_{\x|_1^{\rho(n)}}'(x)\|^{s}e^{\alpha s n}.
	\end{split}
	\]
	To obtain the penultimate line of the above we employ the bounded distortion property \eqref{eq:boundeddist} and to obtain the final inequality we use the chain rule.
	
	Applying Proposition~\ref{prop:massdist}, let $\eta$ be the probability measure described, let $\delta>0$ be arbitrary, and let $K$ be the corresponding index given in Proposition~\ref{prop:massdist}. Choose $R>0$ sufficiently small such that $\min\{|\ii|:\ii\in\Theta_r\}\geq K$ for every $0<r<R$.
	
	By \cite[Corollary~5.8 and Theorem~3.9]{KaenVil}, there exists a constant $C\geq 0$ such that for any bounded Borel subset $I$ of $\R^t$
	$$
	\sharp\{\ii\in\Theta_{\mathrm{diam}(I)}:f_\ii(X)\cap I\neq\emptyset\}\leq C.
	$$
	
	Let $I$ be a bounded Borel subset of $\R^t$ such that $\diam(I)<R$. Now, by Proposition~\ref{prop:massdist}, we have
	$$
	\pi_*\eta(I)\leq\pi_*\eta\left(\bigcup_{\substack{\ii\in\Theta_{\mathrm{diam}(I)}\\f_\ii(X)\cap I\neq\emptyset}}f_\ii(X)\right)\leq \sum_{\substack{\ii\in\Theta_{\mathrm{diam}(I)}\\f_\ii(X)\cap I\neq\emptyset}}\eta([\ii])\ll \sum_{\substack{\ii\in\Theta_{\mathrm{diam}(I)}\\f_\ii(X)\cap I\neq\emptyset}} \delta \diam(X_{\ii})^{s}\leq C\delta\diam(I)^{s}.
	$$
	Let $\{I_i\}_i$ be such that $\pi\widehat{W}(\x,\rho)\subseteq\bigcup_iI_i$ and $\diam(I_i)<R$, i.e. let $\{I_i\}_i$ be an $R$-cover for $\pi\widehat{W}(\x,\rho)$. Also recall that, by Proposition \ref{prop:massdist}, we have $\eta(\widehat{W}(\x,\rho))=1$. Hence, using the above inequality, we have
	$$
	\sum_i\diam(I_i)^{s}\gg\sum_i\frac{1}{\delta}\pi_*\eta(I_i)\geq\frac{1}{\delta}\pi_*\eta\left(\bigcup_iI_i\right)\geq\frac{1}{\delta}\pi_*\eta(\pi\widehat{W}(\x,\rho))=\frac{1}{\delta}.
	$$
	Therefore, by the definition of Hausdorff $s$-measure, $\mathcal{H}^{s}(\pi\widehat{W}(\x,\rho))\gg\frac{1}{\delta}$. Since $\delta>0$ was arbitrary, this implies that $\mathcal{H}^{s}(\pi\widehat{W}(\x,\rho))=\infty$. By Lemma~\ref{lem:incl}, $\mathcal{H}^{s}(\pi\widehat{W}(\x,\rho))\leq \mathcal{H}^{s}(W(x,\Psi))$, and thus the proof is complete.
\end{proof}

Before we turn to the proof of Proposition~\ref{prop:massdist}, we prove a technical lemma.
We say that $\ii=(i_1,i_2,\ldots)\in\Sigma$ is \emph{$m$-periodic} if $i_k=i_{k+m}$ for every $k\geq1$. We say that $\ii \in \Sigma$ is \emph{$m$-periodic on $(\ell,n)$}, where $n-\ell\geq m$, if $i_k=i_{k+m}$ for $\ell\leq k\leq n-m$.

\begin{lemma}\label{lem:avoidcollision1}
	Let $\x\in\Sigma$ and let $n\in\N$. Suppose that $$
	m(\x,n):=\min\{k:1\leq k\text{ and }\x|_1^{n-k}=\x|_{k+1}^n\}<n/2.
	$$
	Then, $\x$ is $m(\x,n)$-periodic on $(1,n)$. Moreover, for each $1\leq k\leq n-m(\x,n)$, we have $\x|_1^{n-k}=\x|_{k+1}^n$ if and only if there exists $p\in\N$ such that $k=p\cdot m(\x,n)$.
\end{lemma}

\begin{proof}
	For convenience, let us write $m=m(\x,n)$. By the definition of $m$, we have \mbox{$x_\ell=x_{m+\ell}$} for every $\ell=1,\ldots,n-m$ and, thus, the proof of the first part of the lemma is complete.
	
	Next, let $q :=\left\lfloor\frac{n}{m}\right\rfloor$. Note that $q \geq 2$ since $m  < \frac{n}{2}$. By using the $m$-periodicity of $\x$, we have that $x_\ell=x_{qm+\ell}$ for every \mbox{$\ell=1,\ldots,n-qm$}. In other words, $\x|_1^{n-qm}=\x|_{qm+1}^n$. Thus, again using the $m$-periodicity of $\x$, there are words $\tv\in\Lambda^{n-qm}$ and $\ov\in\Lambda^{(q+1)m-n}$ such that $\x=\tv\ov\tv\ldots\ov\tv$. Hence, for every $p=1,\ldots,q$, we have $\x|_1^{n-pm}=\x|_{pm+1}^n$. In particular, this yields that $\x|_1^{n-k}=\x|_{k+1}^n$ if $k=pm$ for some $p=1,\ldots,q$.
	
	For the other direction, we argue by contradiction. Let us suppose that there exists some $k$ such that $m\nmid k$ and $\x|_1^{n-k}=\x|_{k+1}^n$. By the definition of $m$, it follows that $m<k$. In order to obtain a contradiction, it is enough to show that
	$$
	\x|_{1}^{n-k+m}=\x|_{k+1-m}^{n}.$$
	Then, by induction, one can find $\ell\in\N$ such that for $k'=k-\ell m<m$, we have \mbox{$\x|_{k'+1}^n=\x|_1^{n-k'}$}, which is a contradiction.
	
	Since $\x|_{n-m+1}^n=\ov\tv$ and $\x|_1^{n-k}=\x|_{k+1}^{n}$, by using the $m$-periodicity of $\x$ we have $\ov\tv=\x|_{n-m+1}^n=\x|_{n-k-m+1}^{n-k}=\x|_{n-k+1}^{n-k+m}$. Similarly, $\tv\ov=\x|_1^m=\x|_{k+1}^{k+m}=\x|_{k+1-m}^{k}$. Hence, using the $m$-periodicity of $\x$ and the fact that $\x|_{1}^{m}=\tv\ov$, we have
	\[
	\begin{split}
	\x|_1^{n-k+m}&=\x|_1^{n-k}\ov\tv=\tv\ov\x|_{m+1}^{n-k}\ov\tv=\tv\ov\x|_1^{n-k-m}\ov\tv\\
	&=\tv\ov\x|_{k+1}^{n-m}\ov\tv=\tv\ov\x|_{k+1}^n=\x|_{k-m+1}^n,
	\end{split}
	\]
as required.
\end{proof}

\section{Construction of the Mass Distribution} \label{mass distribution section}

Let $\Phi=\{f_i\}_{i\in\Lambda}$ be a conformal iterated function system satisfying the open set condition. Throughout the next three sections, we fix an $\x\in\Sigma$ and a function $\rho\colon\N\mapsto\N$ such that $n \mapsto \rho(n)$ is monotonically increasing and for which
\[\liminf_{n\to\infty}\frac{-1}{n}\log \|f_{\x|_1^{\rho(n)}}'(\xi)\|=:\alpha\in(0,\infty),\]
where $\xi \in X$. Note that by the bounded distortion property \eqref{eq:boundeddist}, we may take $\xi$ to be any element of $X$. { We extend the function $\rho\colon\N\mapsto\N$ to a map $\rho\colon\R^+\mapsto\N$ in a natural way; that is, $\rho(x):=\rho(\lfloor x\rfloor)$.}

Let $s$ be the unique solution of the equation $P(s)=s\alpha$ and suppose that
\[\sum_{n=0}^\infty e^{\alpha s n}\|f_{\x|_1^{\rho(n)}}'(\xi)\|^{s}=~\infty.\]

For a strictly monotonic increasing sequence $\A=(A_k)$ of natural numbers, let
$$
C_\A=\{\jj\in\Sigma:\sigma^{\ell}\jj\notin[\x|_1^{\rho(\ell)}]\text{ for every }A_k\neq\ell\geq0\text{ but }\sigma^{A_k}\jj\in[\x|_1^{\rho(A_k)}]\text{ for every }k\geq1\}.
$$

Observe that for every strictly monotonic increasing sequence $\A$, $C_\A$ is compact and for $\A\neq\A'$, $C_\A\cap C_{\A'}=\emptyset$.

In order to achieve the correct dimension (as given by Theorem \ref{thm:hillvelani}), we restrict ourselves to the sequences $\A$, which are rapidly growing. By taking sequences $\A$ which are rapidly growing, we ensure that $C_{\A} \subset \widehat{W}(\x,\rho)$ but at the same time benefits from as much freedom as possible between consecutive ``hits'' of the shrinking target set. The next lemma will be used to show that there exists an uncountable set of such sequences. To save on notation, let us write
\[\varepsilon(n):=e^{\alpha s n}\|f_{\x|_1^{\rho(n)}}'(\xi)\|^{s}.\]

\begin{lemma}\label{lem:existseq}
There exist sequences $\{n_k\}$ and $\{m_k\}$ such that \begin{enumerate}[(i)] \itemsep=5pt
		\item\label{it:b1} $n_1>\max\left\{4,\frac{4}{\alpha^2}\right\}$,
		\item\label{it:b2} $n_1>\left(\dfrac{-4\log a_{\min}+\alpha}{\alpha}\right)^2$,
		\item\label{it:b3} $\displaystyle{n_1>\frac{-8\log a_{\min}}{\alpha}\left(\frac{-2\log a_{\min}}{\alpha}+2\right)}$,
		\item\label{it:b4} $\displaystyle{C\left(a_{\max}^{s}e^{-\alpha s}\right)^{\min\{1,\frac{\alpha}{-2\log a_{\min}}\}\sqrt{1+\frac{\alpha}{-2\log a_{\min}}}\sqrt{n_1}}}\displaystyle{<1},$
	\end{enumerate}
 and for every $n\geq n_1-\sqrt{n_1}$, we have $\rho(n)\geq\frac{\alpha}{-2\log a_{\min}}n$. Moreover, for every $k\geq1$,

	\vbox{\begin{enumerate}[(1)] \itemsep=5pt
		\item\label{it:1} $n_k+\rho(n_k)<m_k$,
		\item\label{it:2} $\max\{m_{k-1}+\rho(m_{k-1})+2,(2m_{k-1}+\rho(m_{k-1}))^2\}<n_{k}$,
		\item\label{it:3} $\displaystyle{\lim_{\ell\to\infty}\dfrac{C^{(1+s)\ell}a_{\min}^{-2s\ell}e^{s\alpha(\sum_{j=1}^{\ell-1}(m_j+\rho(m_j))+2\ell)}}{\prod_{j=1}^\ell\sum_{k=n_j}^{m_j}\varepsilon(k)}=0},$

		\item\label{it:4} $\displaystyle{\lim_{\ell\to\infty}\dfrac{C^{(1+s)\ell}a_{\min}^{-2s\ell}e^{-s\alpha(n_{\ell}-\sum_{j=1}^{\ell-1}(m_j+\rho(m_j))-2\ell)}}{\prod_{j=1}^{\ell-1}\sum_{k=n_j}^{m_j}\varepsilon(k)}=0}$.
	\end{enumerate}}
Throughout this lemma, $C$ is the constant arising from \eqref{eq:boundeddist}.
\end{lemma}

{ Essentially, the  terms \eqref{it:3} and \eqref{it:4} of Lemma~\ref{lem:existseq} will play an important role in the proof of Proposition~\ref{prop:massdist}. The construction of such sequences is possible basically because in \eqref{it:4} $n_\ell$ appears only in the numerator in the term $e^{-s\alpha n_\ell}$ and in \eqref{it:3} $m_\ell$ appears only in the denominator in $\sum_{k=n_\ell}^{m_\ell}\varepsilon_k$. This allows us to choose a sufficiently rapidly growing sequence such that these terms converge to zero.}

\begin{proof}
First of all note that, by the definition of $\alpha$, there exists an $N \in \N$ such that for all natural numbers $n \geq N$,
\[\frac{\alpha}{2} \leq \frac{-\log{\|f'_{\x|_1^{\rho(n)}}(\xi)\|}}{n}\] for any $\xi \in X$. Combining this with the bounds in \eqref{f' composition bounds} we see that
\[\frac{\alpha}{2} \leq \frac{-\log{a_{\min}^{\rho(n)}}}{n} = \frac{-\rho(n)\log{a_{\min}}}{n}.\]
Thus, for all $n \geq N$,
\[\rho(n) \geq \frac{\alpha}{-2\log{a_{\min}}}n.\]

We now construct sequences $\{n_k\}$ and $\{m_k\}$ inductively. Let us fix an arbitrary sequence converging to $0$, say $p_n=2^{-n}$.
We begin by choosing $n_1$ sufficiently large so that \eqref{it:b1}-\eqref{it:b4} hold and $n_1-\sqrt{n_1}\geq N$. Then choose $m_1$ such that $n_1+\rho(n_1)<m_1$. We then proceed by induction. Suppose that $n_k$ and $m_k$ satisfying \eqref{it:1}--\eqref{it:4} have already been defined for $k=1,\ldots,\ell-1$. Next, find $n_\ell$ such that \eqref{it:2} holds and
$$
\dfrac{C^{(1+s)\ell}a_{\min}^{-2s\ell}e^{-s\alpha(n_{\ell}-\sum_{j=1}^{\ell-1}(m_j+\rho(m_j))-2\ell)}}{\prod_{j=1}^{\ell-1}\sum_{k=n_j}^{m_j}\varepsilon(k)}<p_\ell.
$$
This is possible since $s,\alpha>0$. We then find $m_\ell$ so that \eqref{it:1} holds and
$$
\dfrac{C^{(1+s)\ell}a_{\min}^{-2s\ell}e^{s\alpha(\sum_{j=1}^{\ell-1}(m_j+\rho(m_j))+2\ell)}}{\prod_{j=1}^\ell\sum_{k=n_j}^{m_j}\varepsilon(k)}<p_\ell.
$$
This is possible by the divergence of $\sum_{k=1}^{\infty}\varepsilon(k)$. By construction, the sequences $\{n_k\}$ and $\{m_k\}$ satisfy all of the required conditions, thus completing the proof of the lemma.
\end{proof}

Let $\Xi$ be the set of sequences such that $A_k\in[n_k,m_k]$ for every $k\geq1$. In the rest of the paper, we construct the mass distribution $\eta$ as follows. We define a family of probability measures $\{\mu_\A\}_{\A\in\Xi}$, where $\mu_\A$ is supported on $C_\A$, and an appropriate probability measure $\nu$ on $\Xi$. We will then set $\eta=\int\mu_\A d\nu(\A)$. Clearly, $\eta(\widehat{W}(\x,\rho))=1$ since, for every $\A \in \Xi$, $C_{\A} \subset \widehat{W}(\x,\rho)$.

Let us define $R\colon\R^+\mapsto\N$ as follows
$$
R(x):=\max\{m\in\N:x\geq m+\rho(m)\}.
$$
We adopt the convention that $R(n) = 1$ if $n < 1+\rho(1)$.
Since $\rho(n)\geq\frac{\alpha}{-2\log a_{\min}}n$ for every $n\geq n_1-\sqrt{n_1}$, it follows from the definition of $R(n)$ that
\[n \geq R(n)+\rho(R(n)) \geq \left(1+\frac{\alpha}{-2\log{a_{\min}}}\right)R(n) \quad \]
whenever $R(n) \geq n_1-\sqrt{n_1}$. In particular, for every $n \geq n_1-\sqrt{n_1} + \rho(n_1-\sqrt{n_1})$ we have
\begin{equation}\label{eq:ubR}
R(n)\leq\frac{-2\log a_{\min}}{\alpha-2\log a_{\min}}n.
\end{equation}

Let $p+\rho(p)+2<q$ be integers such that $p \geq 0$. Let
\begin{equation}\label{eq:defomega}
\begin{split}
\Omega_{p,q}:=\left\{\ii\in\Lambda^{q-p-2}:\right.&\ii|_{1}^{\ell+\rho(\ell)-p-1}\neq\x|_{p-\ell+2}^{\rho(\ell)}\text{ for every }\ell=R(p+\sqrt{p}+1)+1,\ldots,p\\
&\ii|_{\ell+1}^{\ell+\rho(\ell+p+1)}\neq\x|_1^{\rho(\ell+p+1)}\text{ for }\ell=0,\ldots,R(q-1)-p-1, \text{ and}\\
&\left.\ii|_{\ell+1}^{q-p-2}\neq\x|_1^{q-p-\ell-2}\text{ for }\ell=R(q-1)-p,\ldots,q-\sqrt{q}\right\}.
\end{split}
\end{equation}
{ By the choice of $p$ and $q$ we see that the subwords in the definition of the set $\Omega_{p,q}$ are well defined, since $\ell+\rho(\ell)-p-1\leq \rho(p)-1\leq q-p-2$ and $\ell+\rho(\ell+p+1)\leq R(q-1)-p-1+\rho(R(q-1))\leq q-p-2$ and $R(q-1)-p\geq0$ by the definition of $R$. We will later show that the set $\Omega_{p,q}$ is non-empty by giving a lower bound on its measure. }

Let us recall that
$$
m(\x,n)=\min\{k:1\leq k\text{ and }\x|_1^{n-k}=\x|_{k+1}^n\}.
$$
Observe that if $\x|_1^{n-m(\x,n)}=\x|_{m(\x,n)+1}^n$ then $\x|_1^{n-m(\x,n)-1}=\x|_{m(\x,n)+1}^{n-1}$, and hence,
\begin{equation}\label{eq:mmon}
m(\x,n-1)\leq m(\x,n)\text{ for every }n\geq3.
\end{equation}

Next, we define two functions $\ov,\tv\colon\Xi\mapsto\Sigma$. For $\A \in \Xi$, the $k$th coordinates of $\ov$ and $\tv$ will depend solely upon on the $k$th coordinate $A_k$ of $\A$, we will use the notation
\begin{eqnarray*}
&\ov(\A) &= (\omega_1(A_1), \omega_2(A_2), \dots, \omega_k(A_k), \dots), \\
\text{and} \quad &\tv(\A) &= (\tau_1(A_1), \tau_2(A_2), \dots, \tau_k(A_k),\dots).
\end{eqnarray*} For a sequence $\A\in\Xi$, we define the $k$th coordinate of $\ov$ as to be arbitrary but {\it not} equal to $x_{m(\x,\rho(A_k-\sqrt{A_k}))}$. Similarly, for $\A\in\Xi$, define $\tau_k(A_k)$ as to be arbitrary but {\it not} equal to $x_{\rho(A_k)-\left\lfloor\frac{\rho(A_k)}{m(\x,\rho(A_k))}\right\rfloor\cdot m(\x,\rho(A_k))+1}$.

\begin{lemma}\label{lem:gamset}
	Let $y\in\Lambda$ be such that $y\neq x_1$ and let $\y=(y,y,\ldots)$. For an $\A\in\Xi$, let
	\begin{equation}\label{eq:gamma}
	\Gamma_\A:=[\y|_1^{A_1-1}]\times\prod_{k=1}^\infty\left(\{\omega_k(A_k)\x|_{1}^{\rho(A_k)}\tau_k(A_k)\}\times\Omega_{A_k+\rho(A_k),A_{k+1}}\right).
	\end{equation}
Then $\Gamma_\A\subseteq C_\A$ for every $\A\in\Xi$.
\end{lemma}

Note that $\Gamma_\A\subseteq\widehat{W}(\x,\rho)$ for every $\A\in\Xi$ clearly by definition. However, later in the proof we need the more technical containment $\Gamma_\A\subseteq C_\A$, since then for any $\ii\in\bigcup_{\A\in\Xi}\Gamma_\A$ there exists a unique $\A\in\Xi$ such that $\ii\in\Gamma_\A$, and this is heavily used during the proof of Proposition~\ref{prop:massdist}.

\begin{proof}[Proof of Lemma~\ref{lem:gamset}]
    Observe that if $\ii=(i_1,i_2,\ldots)\in\Gamma_\A$, then $\ii|_1^{A_1-1}=\y|_1^{A_1-1}$ and, for every $k\geq1$, we have:
\begin{itemize} \itemsep=5pt
\item{$\displaystyle{i_{A_k}=\omega_k(A_k)}$,}
\item{$\displaystyle{i_{A_k+\rho(A_k)+1}=\tau_{k}(A_k)}$,}
\item{$\displaystyle{\ii|_{A_k+1}^{A_k+\rho(A_k)}=\x|_1^{\rho(A_k)}}$, and}
\item{$\displaystyle{\ii|_{A_k+\rho(A_k)+2}^{A_{k+1}-1}\in\Omega_{A_k+\rho(A_k),A_{k+1}}}$.}
\end{itemize}
In particular, $\sigma^{A_k}\ii\in[\x|_1^{\rho(A_k)}]$ for every $k\geq1$.
	
    Now let $\ell\notin\A$. To show that $\ii \in C_{\A}$, we need to show that $\sigma^\ell\ii\notin[\x|_1^{\rho(\ell)}]$. We argue by contradiction, so suppose to the contrary that $\sigma^\ell\ii\in[\x|_1^{\rho(\ell)}]$. It is easy to see that $\ell<A_1$ is not possible since $y\neq x_1$. Hence, we may assume that $A_k<\ell<A_{k+1}$ for some $k\geq1$.    
    
    There are five possible cases to consider:
    \begin{enumerate}
    \item\label{it:c0a} $A_k<\ell\leq R\left(A_k+\rho(A_k)+\left\lfloor\sqrt{A_k+\rho(A_k)}\right\rfloor+1\right)$
    \item\label{it:c1} $R\left(A_k+\rho(A_k)+\left\lfloor\sqrt{A_k+\rho(A_k)}\right\rfloor+1\right)+1\leq\ell\leq A_{k}+\rho(A_k)$
    \item\label{it:c2} $A_{k}+\rho(A_k)+1\leq\ell\leq R(A_{k+1}-1)$,
    \item\label{it:c3} $R(A_{k+1}-1)+1\leq\ell\leq A_{k+1}-1-\left\lfloor\sqrt{A_{k+1}}\right\rfloor$,
    \item\label{it:c0b} $A_{k+1}-\left\lfloor\sqrt{A_{k+1}}\right\rfloor\leq\ell< A_{k+1}$
    \end{enumerate}

First, consider the case \eqref{it:c0a}. Recall that $\sigma^{A_k}\ii\in[\x|_1^{\rho(A_k)}]$ and we are assuming, in order to eventually reach a contradiction, that $\sigma^{\ell}\ii\in[\x|_1^{\rho(\ell)}]$. Consequently, 
\begin{align*}
\ii|_{A_k+1}^{A_k+\rho(A_k)} = \x|_{1}^{\rho(A_k)} \qquad \text{and} \qquad
\ii|_{\ell+1}^{\ell+\rho(\ell)} = \x|_{1}^{\rho(\ell)}.
\end{align*}
So, it follows that
\[\x|_{\ell-A_k+1}^{\rho(A_k)} = \ii|_{A_k+1+(\ell-A_k)}^{A_{k}+\rho(A_k)} = \ii|_{\ell + 1}^{\rho(A_k)+A_k} = \x|_{1}^{\rho(A_k)+A_k-(\ell+1)+1} = \x|_{1}^{\rho(A_k)-\ell+A_k}.\]
Thus, we have
\begin{align*}
	m(\x,\rho(A_k))&\leq\ell-A_k \\
	&\leq R\left(A_k+\rho(A_k)+1+\left\lfloor\sqrt{A_k+\rho(A_k)}\right\rfloor\right)-A_k\\
	&\leq \rho(A_k) +1 +\sqrt{A_k+\rho(A_k)}-\rho\left(R\left(A_k+\rho(A_k)+1+\left\lfloor\sqrt{A_k+\rho(A_k)}\right\rfloor\right)\right)\\
	&\leq\rho(A_k) +1 +\sqrt{A_k+\rho(A_k)}-\rho\left(R\left(A_k+\rho(A_k)\right)\right)\\
	&=\rho(A_k) +1 +\sqrt{A_k+\rho(A_k)}-\rho\left(A_k\right)\\
	&=\sqrt{A_k+\rho(A_k)}+1.
\end{align*}

By Lemma~\ref{lem:existseq}\eqref{it:b3} and the fact that $\rho(A_k)\geq\frac{\alpha A_k}{-2\log{a_{\min}}}$, we have
\[A_k<\frac{\alpha}{-2\log a_{\min}}A_k\left(\frac{\alpha}{-8\log a_{\min}}A_k-2\right)<\rho(A_k)\left(\frac{\rho(A_k)}{4}-2\right)+1.\]
Hence, $m(\x,\rho(A_k))\leq\ell-A_k\leq\sqrt{A_k+\rho(A_k)}+1<\rho(A_k)/2\leq\rho(A_k)-m(\x,\rho(A_k))$. So, it follows from Lemma~\ref{lem:avoidcollision1} that there exists $p\in\N$ such that $\ell-A_k=p\cdot m(\x,\rho(A_k))$ and also that $\x$ is $m(\x,\rho(A_k))$-periodic on $(1,\rho(A_k))$. However, since $\ii|_{\ell+1}^{\ell+\rho(\ell)}=\x|_1^{\rho(\ell)}$, we have
\begin{multline*}\tau_k(A_k)=i_{A_k+\rho(A_k)+1}=i_{A_k-\ell+\ell+\rho(A_k)+1}=i_{\rho(A_k)-p\cdot m(\x,\rho(A_k))+\ell+1}\\
	=x_{\rho(A_k)-p\cdot m(\x,\rho(A_k))+1}=x_{\rho(A_k)-\left\lfloor\frac{\rho(A_k)}{m(\x,\rho(A_k))}\right\rfloor\cdot m(\x,\rho(A_k))+1}\neq\tau_k(A_k),
\end{multline*}
which contradicts the definition of $\tau_k(A_k)$.

Suppose that \eqref{it:c1} holds. Since $\sigma^\ell\ii\in[\x|_1^{\rho(\ell)}]$, we have $\ii|_{A_k+\rho(A_k)+2}^{\ell+\rho(\ell)}=\x|_{A_k+\rho(A_k)-\ell+2}^{\rho(\ell)}$. However, $\ii|_{A_k+\rho(A_k)+2}^{A_{k+1}-1}\in\Omega_{A_k+\rho(A_k),A_{k+1}}$ and so, by definition, we have $\ii|_{A_k+\rho(A_k)+2}^{\ell+\rho(\ell)}\neq\x|_{A_k+\rho(A_k)-\ell+2}^{\rho(\ell)}$, which is a contradiction.

Similarly, if \eqref{it:c2} holds then again by definition $\ii|_{\ell+1}^{\ell+\rho(\ell)}\neq\x|_1^{\rho(\ell)}$, and hence $\sigma^\ell\ii\notin[\x|_1^{\rho(\ell)}]$, a contradiction.

If \eqref{it:c3} holds then, by $\sigma^\ell\ii\in[\x|_1^{\rho(\ell)}]$, we have $\ii|_{\ell+1}^{A_{k+1}-1}=\x|_1^{A_{k+1}-1-\ell}$. However, it follows from the fact that $\ii|_{A_k+\rho(A_k)+2}^{A_{k+1}-1} \in \Omega_{A_k+\rho(A_k),A_{k+1}}$ that $\ii|_{\ell+1}^{A_{k+1}-(A_k+\rho(A_k))-2} \neq \x|_{1}^{A_{k+1}-(A_k+\rho(A_k))-\ell-2}$. In particular, this implies that $\ii|_{\ell+1}^{A_{k+1}-1} \neq \x_{1}^{A_{k+1}-1-\ell}$, which is again a contradiction.

Finally, let us consider the case \eqref{it:c0b}, which proof is similar to the case \eqref{it:c0a}. Since $\sigma^\ell\ii\in[\x|_1^{\rho(\ell)}]$ and $\sigma^{A_{k+}}\ii\in[\x|_1^{\rho(A_{k+1})}]$, we have $$\x|_{A_{k+1}+1-\ell}^{\rho(\ell)}=\ii|_{A_{k+1}+1}^{\rho(\ell)+\ell}=\x|_1^{\rho(\ell)+\ell-A_{k+1}}.$$
Thus, $m(\x,\rho(\ell))\leq A_{k+1}-\ell\leq\sqrt{A_{k+1}}$. Observe that, by \eqref{eq:mmon}, we have $m(\x,\rho(\ell))\geq m(\x,\rho(A_{k+1}-\lfloor \sqrt{A_{k+1}} \rfloor))$ and also note that $\x|_{A_{k+1}+1-\ell}^{\rho(A_{k+1}-\lfloor \sqrt{A_{k+1}} \rfloor)}=\x|_1^{\rho(A_{k+1}-\lfloor \sqrt{A_{k+1}} \rfloor)+\ell-A_{k+1}}$. By Lemma~\ref{lem:existseq}, since $\A\in\Xi$, we have
	$$
	\rho(A_{k+1}-\lfloor\sqrt{A_{k+1}}\rfloor)\geq \frac{\alpha}{-2\log a_{\min}}\left(A_{k+1}-\sqrt{A_{k+1}}\right).
	$$
By Lemma~\ref{lem:existseq}\eqref{it:b2}, we have $\sqrt{A_{k+1}}<\frac{\alpha}{-4\log a_{\min}}\left(A_{k+1}-\sqrt{A_{k+1}}\right)$ and, hence,
\[
\begin{split}
m(\x,\rho(A_{k+1}-\lfloor \sqrt{A_{k+1}} \rfloor))&\leq A_{k+1}-\ell<\frac{\rho(A_{k+1}-\lfloor \sqrt{A_{k+1}} \rfloor)}{2}\\
&\leq \rho(A_{k+1}-\sqrt{A_{k+1}})-m(\x,\rho(A_{k+1}-\lfloor \sqrt{A_{k+1}} \rfloor).	
\end{split}
\]
Again, it follows from Lemma \ref{lem:avoidcollision1} that $\x$ is $m(\x,\rho(A_{k+1}-\lfloor \sqrt{A_{k+1}} \rfloor))$-periodic on\linebreak \mbox{$(1,\rho(A_{k+1}-\lfloor \sqrt{A_{k+1}} \rfloor))$}. Furthermore, we also have that there exists $q\in\N$ such that\linebreak \mbox{$q\cdot m(\x,\rho(A_{k+1}-\lfloor \sqrt{A_{k+1}} \rfloor))=A_{k+1}-\ell$.} Hence, by the definition of $\omega_{k+1}(A_{k+1})$ and using that $\ii|_{\ell+1}^{\rho(\ell)+\ell}=\x|_1^{\rho(\ell)}$, we have
\begin{align*}
	\omega_{k+1}(A_{k+1})=i_{A_{k+1}}&=i_{A_{k+1}-\ell+\ell}=i_{q\cdot m(\x,\rho(A_{k+1}-\lfloor \sqrt{A_{k+1}} \rfloor))+\ell} \\
	                                       &=x_{q\cdot m(\x, \rho(A_{k+1}-\lfloor \sqrt{A_{k+1}} \rfloor))}=x_{m(\x,\rho(A_{k+1}-\lfloor \sqrt{A_{k+1}} \rfloor))}\neq\omega_{k+1}(A_{k+1}),
\end{align*}
which is a contradiction. \end{proof}

For $\ii \in \Lambda^{q-p-2}$, let
\[\PP_{p,q}(\ii)=\PP([\ii]),\]
where $\PP$ is the measure defined by \eqref{eq:measP}. Thus, $\PP_{p,q}$ is the probability measure supported on $\Lambda^{q-p-2}$ corresponding to the equilibrium state $\PP$.

Denote by $\widehat{\Omega}_{p,q}$ the subset of $\Sigma$ such that $\ii\in\widehat{\Omega}_{p,q}$ if and only if $\ii|_{p+2}^{q-1}\in\Omega_{p,q}$. Since $\PP$ is $\sigma$-invariant by definition, $\PP(\widehat{\Omega}_{p,q})=\PP_{p,q}(\Omega_{p,q})$.

\begin{lemma}\label{lem:boundomega}
There exists a constant $C'>0$ such that for every $p>n_1$ (with $n_1$ as defined in Lemma~\ref{lem:existseq}) and for every $q>\max\{p+\rho(p)+2,(p+2)^2\}$,
	$$
	\PP_{p,q}(\Omega_{p,q})\geq1-C'\left(a_{\max}^{s}e^{-\alpha s}\right)^{\min\left\{\sqrt{p},\sqrt{q}-p,\frac{\alpha p}{-2\log a_{\min}}\right\}}.
	$$
\end{lemma}

\begin{proof}
First, note that $\min\{\rho(\ell+p+1),q-p-2-\ell\}=\rho(\ell+p+1)$ if and only if $\ell \leq R(q-1)-p-1$. Furthermore, note that $a_{\max}^s e^{-\alpha s} < 1$ since $\alpha > 0$ and $s>0$.
	
	Denote the complement of $\widehat{\Omega}_{p,q}$ by $\widehat{\Omega}_{p,q}^c$. { Since $\widehat{\Omega}_{p,q}=\bigcup_{\ii\in\Omega_{p,q}}\sigma^{-(p+1)}[\ii]$, using the definition of $\Omega_{p,q}$ in \eqref{eq:defomega} we have that}
\begin{equation} \label{complement inclusion}
\widehat{\Omega}_{p,q}^{c} \subseteq \left(\bigcup_{\ell=R(p+\sqrt{p}+1)+1}^{p}\sigma^{-(p+1)}\left[\x|_{p-\ell+2}^{\rho(\ell)}\right]\right)\bigcup\left(\bigcup_{\ell=0}^{q-\sqrt{q}}\sigma^{-(\ell +p+1)}\left[\x|_1^{\min\{\rho(\ell+p+1),q-p-2-\ell\}}\right]\right).
\end{equation}

	 Then by the $\sigma$-invariance of $\PP$, for every $q>p+\rho(p)+2$ we have
\[\PP(\widehat{\Omega}_{p,q}^c) \leq \sum_{\ell=R(p+\sqrt{p}+1)+1}^{p}\PP\left(\left[\x|_{p-\ell+2}^{\rho(\ell)}\right]\right)+\sum_{\ell=0}^{q-\sqrt{q}}\PP\left(\left[\x|_1^{\min\{\rho(\ell+p+1),q-p-2-\ell\}}\right]\right). \]
Here, the assumption that $q > p+\rho(p)+2$ ensures that $\left[\x|_1^{\min\{\rho(\ell+p+1),q-p-2-\ell\}}\right] \neq \emptyset$ for $\ell \geq 0$. It then follows from \eqref{eq:measP} and \eqref{f' composition bounds} that
\begin{align*}
	\PP(\widehat{\Omega}_{p,q}^c)&\leq C\sum_{\ell=R(p+\sqrt{p}+1)+1}^{p}\|f_{\x|_{p-\ell+2}^{\rho(\ell)}}'(\xi)\|^se^{-(\rho(\ell)-p+\ell-1)\alpha s}\\
	&+C\sum_{\ell=0}^{q-\sqrt{q}}\|f_{\x|_1^{\min\{\rho(\ell+p+1),q-p-2-\ell\}}}'(\xi)\|^se^{-\min\{\rho(\ell+p+1),q-p-2-\ell\}\alpha s} \\
	&\leq C\sum_{\ell=R(p+\sqrt{p}+1)+1}^{p}\left(a_{\max}^{s}e^{-\alpha s}\right)^{\rho(\ell)+\ell-p-1}+C\sum_{\ell=0}^{q-\sqrt{q}}\left(a_{\max}^{s}e^{-\alpha s}\right)^{\min\{\rho(\ell+p+1),q-p-2-\ell\}}.
\end{align*}

Since $\ell\mapsto\rho(\ell)+\ell-p-1$ is a strictly monotonic increasing sequence of integers and
\[\rho(R(p+\sqrt{p}+1)+1)+R(p+\sqrt{p}+1)+1-p-1>p+\sqrt{p}+1-p-1=\sqrt{p},\]
we have
\begin{align*}
\sum_{\ell=R(p+\sqrt{p}+1)+1}^{p}\left(a_{\max}^{s}e^{-\alpha s}\right)^{\rho(\ell)+\ell-p-1} &\leq\sum_{k=\sqrt{p}}^{\infty}a_{\max}^{k s}e^{-\alpha s k} \\
     &=(a_{\max}^se^{-\alpha s})^{\sqrt{p}}\sum_{k=0}^{\infty}{a_{\max}^{ks}e^{-\alpha s k}} \\
     &=\dfrac{\left(a_{\max}^{s}e^{-\alpha s}\right)^{\sqrt{p}}}{1-a_{\max}^se^{-\alpha s}}.
\end{align*}

Next, note that, by Lemma~\ref{lem:existseq}, since $p > n_1$ by assumption, we have \[\rho(\ell+p+1) \geq \frac{\alpha(\ell+p+1)}{-2\log{a_{\min}}}\]
for any $\ell \geq 0$. Hence, we have
\begin{align*}
	\PP(\widehat{\Omega}_{p,q}^c)&\leq \dfrac{C\left(a_{\max}^{s}e^{-\alpha s}\right)^{\sqrt{p}}}{1-a_{\max}^se^{-\alpha s}}+C\sum_{\ell=R(q-1)-p}^{q-\lfloor\sqrt{q}\rfloor}\left(a_{\max}^{s}e^{-\alpha s}\right)^{q-p-2-\ell}+C\sum_{\ell=0}^{R(q-1)-p-1}\left(a_{\max}^{s}e^{-\alpha s}\right)^{\rho(\ell+p+1)}\\
	&\leq \dfrac{C\left(a_{\max}^{s}e^{-\alpha s}\right)^{\sqrt{p}}}{1-a_{\max}^se^{-\alpha s}}+C\sum_{\ell=\sqrt{q}-p-2}^{q-R(q-1)-2}\left(a_{\max}^{s}e^{-\alpha s}\right)^{\ell}+C\sum_{\ell=0}^{R(q-1)-p-1}\left(a_{\max}^{s}e^{-\alpha s}\right)^{\frac{\alpha(\ell+p+1)}{-2\log a_{\min}}}\\
	&\leq \dfrac{C\left(a_{\max}^{s}e^{-\alpha s}\right)^{\sqrt{p}}}{1-a_{\max}^se^{-\alpha s}}+C\sum_{\ell=\sqrt{q}-p-2}^{\infty}\left(a_{\max}^{s}e^{-\alpha s}\right)^{\ell}+C\sum_{\ell=0}^{\infty}\left(a_{\max}^{s}e^{-\alpha s}\right)^{\frac{\alpha(\ell+p+1)}{-2\log a_{\min}}}\\
	&\leq \dfrac{C\left(a_{\max}^{s}e^{-\alpha s}\right)^{\sqrt{p}}}{1-a_{\max}^se^{-\alpha s}}+\dfrac{C\left(a_{\max}^{s}e^{-\alpha s}\right)^{\sqrt{q}-p-2}}{1-a_{\max}^{s}e^{-\alpha s}}+\dfrac{C\left(a_{\max}^{s}e^{-\alpha s}\right)^{\frac{\alpha(p+1)}{-2\log a_{\min}}}}{1-\left(a_{\max}^{s}e^{-\alpha s}\right)^{\frac{\alpha}{-2\log a_{\min}}}} \\
	&\leq C' \left(a_{\max}^{s}e^{-\alpha s}\right)^{\min\left\{\sqrt{p},\sqrt{q}-p,\frac{\alpha p}{-2\log a_{\min}}\right\}}.
\end{align*}
Note that, since $\ell>0$, we require the condition $q > (p+2)^2$ to ensure the middle term above is valid.

	Finally, since $\PP_{p,q}(\Omega_{p,q})=\PP(\widehat{\Omega}_{p,q})=1-\PP(\widehat{\Omega}_{p,q}^c)$, the statement follows.
\end{proof}

\begin{lemma}\label{lem:prod}
	Let $0<p_n<1$ be a sequence such that $\sum_{n=1}^\infty p_n<\infty$ and $\max_n\{p_n\}<1$. Then $\prod_{n=1}^\infty(1-p_n)>0$.
\end{lemma}

\begin{proof}
	Using the Taylor expansion  of $\log(1-x)$, we see that
$$\log(1-x)=-\sum_{k=1}^\infty\frac{x^k}{k}\geq-x-\sum_{k=2}^\infty \frac{x^k}{2}=-x-\frac{x^2}{2(1-x)}\geq-x\left(1+\frac{1}{2(1-x)}\right)$$
for every $0<x<1$. Therefore,
	\[
	\prod_{n=1}^\infty(1-p_n)=\prod_{n=1}^\infty e^{\log(1-p_n)}\geq\prod_{n=1}^\infty e^{-p_n\left(1+\frac{1}{2(1-\max_np_n)}\right)}= e^{-\left(1+\frac{1}{2(1-\max_np_n)}\right)\sum_{n=1}^\infty p_n}>0. \qedhere
	\]
\end{proof}

\begin{lemma}\label{lem:boundprodomega}
	There exists a constant $Q>0$ such that for every $\A\in\Xi$,
	$$
	\prod_{j=1}^{\infty}\PP_{A_j+\rho(A_j),A_{j+1}}(\Omega_{A_j+\rho(A_j),A_{j+1}})\geq Q.
	$$
\end{lemma}

\begin{proof}
	By Lemma~\ref{lem:existseq}\eqref{it:1} and Lemma~\ref{lem:existseq}\eqref{it:2},
\[\sqrt{n_{j+1}}-m_j-\rho(m_j)\geq m_j\geq n_j+\rho(n_j)\geq\sqrt{n_j+\rho(n_j)}\geq\sqrt{\left(1+\dfrac{\alpha}{-2\log a_{\min}}\right)n_j}.\]
Since $A_j\in[n_j,m_j]$,
	\begin{align*}
	\min&\left\{\sqrt{A_j+\rho(A_j)},\sqrt{A_{j+1}}-A_j-\rho(A_j),\frac{\alpha (A_j+\rho(A_j))}{-2\log a_{\min}}\right\}\\
	&\geq\min\left\{\sqrt{n_j+\rho(n_j)},\sqrt{n_{j+1}}-m_j-\rho(m_j),\frac{\alpha (n_j+\rho(n_j))}{-2\log a_{\min}}\right\}\\
	&\geq\min\left\{\sqrt{\left(1+\dfrac{\alpha}{-2\log a_{\min}}\right)n_j},\frac{\alpha}{-2\log a_{\min}}\left(1+\dfrac{\alpha}{-2\log a_{\min}}\right)n_j\right\} \\
	&\geq\min\left\{1,\frac{\alpha}{-2\log a_{\min}}\right\}\sqrt{\left(1+\dfrac{\alpha}{-2\log a_{\min}}\right)n_j}.
	\end{align*}
	Hence, by Lemma~\ref{lem:boundomega},
\begin{equation}\label{eq:tob1}
		\PP_{A_j+\rho(A_j),A_{j+1}}(\Omega_{A_j+\rho(A_j),A_{j+1}})\geq 1-C\left(a_{\max}^{s}e^{-\alpha s}\right)^{\min\left\{1,\frac{\alpha}{-2\log a_{\min}}\right\}\sqrt{\left(1+\frac{\alpha}{-2\log a_{\min}}\right)n_j}}. \\[3ex]
	\end{equation}
By \eqref{eq:tob1}, we have
	\[
	\begin{split}
	\prod_{j=1}^{\infty}\PP_{A_j+\rho(A_j),A_{j+1}}(\Omega_{A_j+\rho(A_j),A_{j+1}})&\geq \prod_{j=1}^\infty\left(1-C\left(a_{\max}^{s}e^{-\alpha s}\right)^{\sqrt{n_j}\min\left\{1,\frac{\alpha}{-2\log a_{\min}}\right\}\sqrt{1+\frac{\alpha}{-2\log a_{\min}}}}\right)\\
	&\geq \prod_{n=n_1}^\infty\left(1-C\left(a_{\max}^{s}e^{-\alpha s}\right)^{\sqrt{n}\min\{1,\frac{\alpha}{-2\log a_{\min}}\}\sqrt{1+\frac{\alpha}{-2\log a_{\min}}}}\right),
	\end{split}
	\]
	which is a positive constant by Lemma~\ref{lem:existseq}\eqref{it:b4} and Lemma~\ref{lem:prod}.
\end{proof}

Let
$$
\mathbb{S}_{p,q}:=\frac{\mathbb{P}_{p,q}|_{\Omega_{p,q}}}{\PP_{p,q}(\Omega_{p,q})}.
$$

For a sequence $\A\in\Xi$, let $\mathbb{Q}_{k}$ be the probability measure on $\Lambda^{\rho(A_k)+2}$ such that
$$
\mathbb{Q}_k(\ii)=\begin{cases}
1 & \text{if }\ii=\omega_k(A_k)\x|_1^{\rho(A_k)}\tau_k(A_k),\\
0 & \text{otherwise;}
\end{cases}
$$
and let $\delta$ be the probability measure on $\Lambda^{A_1-1}$ such that
$$
\delta(\ii)=\begin{cases}
	1 & \text{if }\ii=\y|_1^{A_1-1},\\
	0 & \text{otherwise.}
\end{cases}
$$
For each $\A\in\Xi$, we define a probability measure $\mu_{\A}$ as follows:
\begin{equation}\label{eq:muA}
\mu_\A=\delta\times\prod_{k=1}^{\infty}\mathbb{Q}_k\times\Sm_{A_k+\rho(A_k),A_{k+1}}.
\end{equation}

It follows from Lemma~\ref{lem:boundprodomega} that $\mu_\A$ is a well-defined probability measure on $\Sigma$ with respect to the $\sigma$-algebra generated by the cylinder sets, since Lemma~\ref{lem:boundprodomega} guarantees that the normalising factor in the definition of $\mu_\A$ is non-zero. Moreover, by construction, $\mu_\A(\Gamma_\A)=1$, where $\Gamma_\A$ is the set defined in Lemma~\ref{lem:gamset}.

Recall that we defined $\varepsilon(n):=e^{\alpha s n}\|f_{\x|_1^{\rho(n)}}'(\xi)\|^{s}$ and assumed that $\sum_{k=1}^\infty\varepsilon(n)$ is a divergent series. Finally, we define the probability measure $\nu$ on $\Xi$ as
\begin{equation}\label{eq:nu}
\nu([A_1,\ldots,A_\ell])=\prod_{j=1}^\ell\frac{\varepsilon(A_j)}{\sum_{k=n_j}^{m_j}\varepsilon(k)}.
\end{equation}

By its construction, the measure $\eta:=\int\mu_\A d\nu(\A)$ is a well-defined probability measure on $\Sigma$ with respect to the $\sigma$-algebra generated by the cylinder sets. We conclude this section by observing that $\eta(\widehat{W}(\x,\rho))=1$. In fact, we actually have the stronger statement that $\eta\left(\bigcup_{\A\in\Xi}\Gamma_\A\right)=1$ since
\[\int_{\Xi}\mu_\A(\bigcup_{\A\in\Xi}\Gamma_\A) d\nu(\A)=\int_{\Xi}\mu_\A(\Gamma_\A) d\nu(\A)=\int_{\Xi}1d\nu(\A)=1.\]
The conclusion that $\eta(\widehat{W}(\x,\rho))=1$ follows from the fact that $\eta\left(\bigcup_{\A\in\Xi}\Gamma_\A\right)=1$ upon recalling that $\Gamma_{\A} \subseteq C_{\A} \subseteq \widehat{W}(\x,\rho)$ for every $\A \in \Xi$.

\section{Proof of Proposition~\ref{prop:massdist}} \label{end of proof section}

Before we turn to the proof, we give estimates for $\mu_\A([\ii|_1^k])$.

\begin{lemma}\label{lem:meas1}
	For every $\A\in\Xi$ and every $\ii\in\Gamma_\A$,
	$$
	\mu_\A([\ii|_1^k])\ll\begin{cases}
	\|f_{\ii|_1^k}'(\xi)\|^s\dfrac{C^{(1+s)\ell}a_{\min}^{-2\ell s}e^{-s\alpha(A_{\ell}-\sum_{j=1}^{\ell-1}\rho(A_j)-2\ell)}}{\|f_{\x|_1^{k-A_\ell}}'(x)\|^{s}\prod_{j=1}^{\ell-1}\|f_{\x|_1^{\rho(A_j)}}'(x)\|^{s}} & \text{ if }A_\ell\leq k<A_\ell+\rho(A_\ell), \\
	\|f_{\ii|_1^k}'(\xi)\|^s\dfrac{C^{(1+s)\ell}a_{\min}^{-2\ell s}e^{-s\alpha(k-\sum_{j=1}^{\ell}\rho(A_j)-2\ell)}}{\prod_{j=1}^{\ell}\|f_{\x|_1^{\rho(A_j)}}'(x)\|^{s}} & \text{ if }A_\ell+\rho(A_\ell)\leq k<A_{\ell+1},
	\end{cases}
	$$
where $C \geq 1$ is the constant taken here to be the maximum of the constants appearing in \eqref{eq:boundeddist} and \eqref{eq:measP}.
\end{lemma}

Note that it follows from Lemma \ref{lem:existseq}(1) and \ref{lem:existseq}(2) that $A_{\ell}-\sum_{j=1}^{\ell-1}{\rho(A_j)}-2\ell>0$ and so the exponent of $e$ is negative.

\begin{proof}
	First suppose that $A_\ell\leq k<A_\ell+\rho(A_\ell)$. Then $\ii|_{A_{\ell}}^k=\omega_\ell(A_\ell)\x|_1^{k-A_\ell}$ (see proof of Lemma \ref{lem:gamset}) and thus,
	\[
	\mu_\A([\ii|_1^k])=\prod_{j=1}^{\ell-1}\Sm_{A_j+\rho(A_j),A_{j+1}}([\ii|_{A_j+\rho(A_j)+2}^{A_{j+1}-1}]).
	\]
	By Lemma~\ref{lem:boundprodomega} we have $\prod_{j=1}^{\ell-1}\PP_{A_j+\rho(A_j),A_{j+1}}(\Omega_{A_j+\rho(A_j),A_{j+1}})\geq Q>0$, and hence
	$$
	\mu_\A([\ii|_1^k])\leq Q^{-1}\prod_{j=1}^{\ell-1}\PP_{A_j+\rho(A_j),A_{j+1}}([\ii|_{A_j+\rho(A_j)+2}^{A_{j+1}-1}]).
	$$
	By \eqref{eq:measP} and since $A_1\leq m_1$ and \eqref{f' bounds},
	\[
	\mu_\A([\ii|_1^k])\leq \frac{C^{\ell}\|f_{\ii|_1^{A_1-1}}'(\xi_0)\|^se^{-\alpha s (A_1-1)}}{Q(a_{\min}^se^{-\alpha s})^{m_1}}\prod_{j=1}^{\ell-1}\|f_{\ii|_{A_j+\rho(A_j)+2}^{A_{j+1}-1}}'(\xi_j)\|^s e^{-\alpha s(A_{j+1}-A_j-\rho(A_j)-2)},
	\]
where $\xi_0,\cdots,\xi_{\ell-1}\in X$ are arbitrary. So one can choose $\xi_{j}=f_{\ii|^k_{A_j-1}}'(x)$, for $j=0\ldots,\ell-1$ and thus, by the chain rule and \eqref{eq:boundeddist}, we have
\begin{align*}
\|f_{\ii|_1^{A_1-1}}'(\xi_0)\|&\prod_{j=1}^{\ell-1}\|f_{\ii|_{A_j+\rho(A_j)+2}^{A_{j+1}-2}}'(\xi_j)\|\\
&=\|f_{\ii|_1^k}'(x)\|\|f_{\omega_\ell(A_\ell)\x|_1^{k-A_\ell}}'(x)\|^{-1}\prod_{j=1}^{\ell-1}\|f_{\omega_j(A_j)\x|_1^{\rho(A_j)}\tau_j(A_j)}'(f_{\ii|_{A_j+\rho(A_j)+2}^k}(x))\|^{-1}\\
&\leq \|f_{\ii|_1^k}'(x)\|C^{\ell-1}a_{\min}^{-(2\ell-1)}\|f_{\x|_1^{k-A_\ell}}'(x)\|^{-1}\prod_{j=1}^{\ell-1}\|f_{\x|_1^{\rho(A_j)}}'(x)\|^{-1},
\end{align*}
where in the last inequality we used \eqref{eq:boundeddist} and \eqref{f' composition bounds}. Thus,

\[
\mu_\A([\ii|_1^k])\ll \|f_{\ii|_1^k}'(\xi)\|^s\cdot C^{(1+s)\ell}a_{\min}^{-2\ell s} e^{-s\alpha(A_{\ell}-\sum_{j=1}^{\ell-1}\rho(A_j)-2\ell)}\|f_{\x|_1^{k-A_\ell}}'(x)\|^{-s}\prod_{j=1}^{\ell-1}\|f_{\x|_1^{\rho(A_j)}}'(x)\|^{-s}.
\]
	Now, consider the case when $A_\ell+\rho(A_\ell)\leq k<A_{\ell+1}$. Then, similarly to the previous argument, we have
	\[
	\begin{split}
	\mu_\A([\ii|_1^k])&=\left(\prod_{j=1}^{\ell-1}\Sm_{A_j+\rho(A_j),A_{j+1}}([\ii|_{A_j+\rho(A_j)+2}^{A_{j+1}-1}])\right)\cdot\Sm_{A_\ell+\rho(A_\ell),A_{\ell+1}}([\ii|_{A_\ell+\rho(A_\ell)+2}^k])\\
	&\ll\|f_{\ii|_1^k}'(\xi)\|^s\cdot C^{(1+s)\ell}a_{\min}^{-2s\ell} e^{-s\alpha(k-\sum_{j=1}^{\ell}\rho(A_j)-2\ell)}\prod_{j=1}^{\ell}\|f_{\x|_1^{\rho(A_j)}}'(x)\|^{-s}.
	\end{split}
	\]
\end{proof}

\begin{lemma}\label{lem:upperbound}
	Let $\A\in\Xi$ and $\ii\in\Gamma_\A$ be arbitrary. Then, for every $k\geq n_1$,
	$$
	\mu_\A([\ii|_1^k])\ll\begin{cases}
	\|f_{\ii|_1^k}'(\xi)\|^s\dfrac{C^{(1+s)\ell}a_{\min}^{-2\ell s}e^{-s\alpha(A_\ell-\sum_{j=1}^{\ell-1}\rho(A_j)-2\ell)}}{\prod_{j=1}^{\ell}\|f_{\x|_1^{\rho(A_j)}}'(x)\|^{s}} & \text{ if }A_\ell\leq k<n_{\ell+1} \\
	\|f_{\ii|_1^k}'(\xi)\|^s\dfrac{C^{(1+s)\ell}a_{\min}^{-2\ell s}e^{-s\alpha(n_{\ell+1}-\sum_{j=1}^{\ell}\rho(A_j)-2\ell)}}{\prod_{j=1}^{\ell}\|f_{\x|_1^{\rho(A_j)}}'(x)\|^{s}} & \text{ if }n_{\ell+1}\leq k<A_{\ell+1},
	\end{cases}
	$$
where $C \geq 1$ is as in Lemma \ref{lem:meas1}.
\end{lemma}

\begin{proof}
	Let $\ell\geq1$ be such that $A_\ell\leq k<A_{\ell+1}$. If $A_\ell\leq k<n_{\ell+1}$ then either $k< A_\ell+\rho(A_\ell)$ or $k\geq A_\ell+\rho(A_\ell)$. If $A_\ell\leq k<A_\ell+\rho(A_\ell)$ then, by Lemma~\ref{lem:meas1},
	\[
	\mu_\A([\ii|_1^k])\ll 	\|f_{\ii|_1^k}'(\xi)\|^s C^{(1+s)\ell}a_{\min}^{-2\ell s}e^{-s\alpha(A_{\ell}-\sum_{j=1}^{\ell-1}\rho(A_j)-2\ell)}\cdot\prod_{j=1}^{\ell}\|f_{\x|_1^{\rho(A_j)}}'(x)\|^{-s}.
	\]
	
	If $A_\ell+\rho(A_\ell)\leq k<n_{\ell+1}$ then, by Lemma~\ref{lem:meas1} again,
\[
\begin{split}
	\mu_\A([\ii|_1^k])&\ll \|f_{\ii|_1^k}'(\xi)\|^s\dfrac{C^{(1+s) \ell}a_{\min}^{-2 \ell s}e^{-s\alpha(k-\sum_{j=1}^{\ell}\rho(A_j)-2\ell)}}{\prod_{j=1}^{\ell}\|f_{\x|_1^{\rho(A_j)}}'(x)\|^{s}}\\
	&\leq \|f_{\ii|_1^k}'(\xi)\|^s\dfrac{C^{(1+s)\ell}a_{\min}^{-2\ell s}e^{-s\alpha(A_\ell-\sum_{j=1}^{\ell-1}\rho(A_j)-2\ell)}}{\prod_{j=1}^{\ell}\|f_{\x|_1^{\rho(A_j)}}'(x)\|^{s}}.
	\end{split}
	\]
	On the other hand, if $n_{\ell+1}\leq k<A_{\ell+1}$ then $k\geq n_{\ell+1}>m_{\ell}+\rho(m_{\ell})\geq A_{\ell}+\rho(A_{\ell})$. Thus, in this case, by Lemma~\ref{lem:meas1},
	\[
	\begin{split}
	\mu_\A([\ii|_1^k])&\ll \|f_{\ii|_1^k}'(\xi)\|^s\dfrac{C^{(1+s)\ell}a_{\min}^{-2\ell s}e^{-s\alpha(k-\sum_{j=1}^{\ell}\rho(A_j)-2\ell)}}{\prod_{j=1}^{\ell}\|f_{\x|_1^{\rho(A_j)}}'(x)\|^{s}}\\
	&\leq \|f_{\ii|_1^k}'(\xi)\|^s\dfrac{C^{(1+s)\ell}a_{\min}^{-2\ell s}e^{-s\alpha(n_{\ell+1}-\sum_{j=1}^{\ell}\rho(A_j)-2\ell)}}{\prod_{j=1}^{\ell}\|f_{\x|_1^{\rho(A_j)}}'(x)\|^{s}}.
	\end{split}
	\]
\end{proof}

Finally, we turn to the proof of our main proposition.

\begin{proof}[Proof of Proposition~\ref{prop:massdist}]
	Let $\delta>0$ be arbitrary but fixed. One can find $L\geq1$ such that all the terms in Lemma~\ref{lem:existseq}\eqref{it:3} and Lemma~\ref{lem:existseq}\eqref{it:4} are smaller than $\delta$ for every $\ell\geq L$. Let us choose $K:=n_L+1$.
	
	Let $k\geq K$ and $\jj\in\Lambda^{k}$ be arbitrary. It can be seen that if $[\jj]\cap\bigcup_{\A\in\Xi}\Gamma_\A=\emptyset$ then $\eta([\jj])=0$ and thus, the bound holds trivially. So, without loss of generality, we may assume that $[\jj]\cap\bigcup_{\A\in\Xi}\Gamma_\A\neq\emptyset$ and pick $\ii\in\bigcup_{\A\in\Xi}\Gamma_\A$ such that $\ii|_1^k=\jj$.

There are two possible cases to consider: either $m_\ell\leq k<n_{\ell+1}$ or $n_\ell\leq k<m_\ell$ for some $\ell\geq1$.

\smallskip
\noindent \underline{\emph{Case 1: $m_\ell\leq k<n_{\ell+1}$}}
\smallskip
	
	If $m_\ell\leq k<n_{\ell+1}$ then by the pairwise disjointness of the sets $\{\Gamma_\A\}_{\A\in\Xi}$ there exists a unique sequence $A_1,\ldots, A_\ell$ such that \mbox{$\sigma^{A_j}\ii\in[\x|_1^{\rho(A_j)}]$} for $j=1,\dots,\ell$, and for every sequence $\A^*\in\Xi$ such that $A_{j}^{*}\neq A_j$ for some $j=1,\ldots,\ell$, we have $[\ii|_1^k]\cap \Gamma_{\A^*}=\emptyset$. Moreover, $A_\ell\leq k<n_{\ell+1}$. So, by Lemma~\ref{lem:upperbound},
	\[\begin{split}
	\eta([\ii|_1^k])&=\int_\Xi \mu_{\A'}([\ii|_1^k])d\nu(\A')=\int_{[A_1,\ldots,A_\ell]}\mu_{\A'}([\ii|_1^k])d\nu(\A')\\
	&\ll\|f_{\ii|_1^k}'(\xi)\|^s\dfrac{C^{(1+s)\ell}a_{\min}^{-2\ell s}e^{-s\alpha(A_\ell-\sum_{j=1}^{\ell-1}\rho(A_j)-2\ell)}}{\prod_{j=1}^{\ell}\|f_{\x|_1^{\rho(A_j)}}'(x)\|^{s}}\nu([A_1,\ldots,A_\ell])\\
&= \|f_{\ii|_1^k}'(\xi)\|^s\dfrac{C^{(1+s)\ell}a_{\min}^{-2\ell s}e^{-s\alpha(A_\ell-\sum_{j=1}^{\ell-1}\rho(A_j)-2\ell)}}{\prod_{j=1}^{\ell}\|f_{\x|_1^{\rho(A_j)}}'(x)\|^{s}}
\frac{\prod_{j=1}^{\ell}\|f_{\x|_1^{\rho(A_j)}}'(x)\|^{s}e^{\alpha s A_j}}{\prod_{j=1}^\ell\sum_{k=n_j}^{m_j}\varepsilon(k)}\\
&= \|f_{\ii|_1^k}'(\xi)\|^s\dfrac{C^{(1+s)\ell}a_{\min}^{-2\ell s}e^{s\alpha(\sum_{j=1}^{\ell-1}(A_j+\rho(A_j))+2\ell)}}{\prod_{j=1}^\ell\sum_{k=n_j}^{m_j}\varepsilon(k)} \\
&\leq\delta\cdot\|f_{\ii|_1^k}'(\xi)\|^s.
	\end{split}\]
The final inequality above follows from Lemma~\ref{lem:existseq}\eqref{it:3} and our choice of $\delta$.

\smallskip
\noindent \underline{\emph{Case 2: $n_\ell\leq k<m_\ell$}}
\smallskip
	
	If $n_\ell\leq k<m_{\ell}$ then again by the pairwise disjointness of the sets $\{\Gamma_\A\}_{\A\in\Xi}$ there exists a unique sequence $A_1,\ldots, A_{\ell-1}$ such that \mbox{$\sigma^{A_j}\ii\in[\x|_1^{\rho(A_j)}]$} for $j=1,\dots,\ell-1$, and for every sequence $\A^{*} \in \Xi$ such that $A_{j}^{*}\neq A_j$ for some $j=1,\ldots,\ell-1$, we have $[\ii|_1^k]\cap \Gamma_{\A^{*}}=\emptyset$. Moreover, there is at most one $n_\ell\leq B\leq R(k)$ such that\linebreak $[\ii|_1^k]\cap\Gamma_{A_1,\ldots,A_{\ell-1},B,\ldots}\neq~\emptyset$ and in that case $[\ii|_1^k]\cap\Gamma_{A_1,\ldots,A_{\ell-1},B',\ldots}=~\emptyset$ for any $B'\neq B$. To see this, note that if $B > R(k)$ then, by definition, $\rho(B)+B >k$ and so it is impossible for $\ii|_{1}^{k}$ to contain $\x|_{1}^{\rho(B)}$ completely. Motivated by this, we decompose $[n_\ell,m_\ell)$ into three parts with respect to $k$:
\begin{itemize} \itemsep=5pt
\item{$[n_\ell,R(k)]$,}
\item{$[R(k)+1,k)$, where $\ii|_{1}^{k}$ may contain part, but not all, of $\x|_{1}^{\rho(B)}$, and}
\item{$[k,m_\ell)$, in which case $\ii|_{1}^{k}$ does not contain any of $\x|_{1}^{\rho(B)}$.}
\end{itemize}

First, suppose that there exists a $B\in[n_\ell,R(k)]$ such that $[\ii|_1^k]\cap\Gamma_{A_1,\ldots,A_{\ell-1},B,\ldots}\neq~\emptyset$. Then, by Lemma~\ref{lem:upperbound},
\[
\begin{split}
\eta([\ii|_1^k])&=\int_{\Xi}\mu_{\A'}([\ii|_1^k])d\nu(\A') \\
&=\int_{[A_1,\ldots,A_{\ell-1},B]}\mu_{\A'}([\ii|_1^k])d\nu(\A')\\
&\ll\|f_{\ii|_1^k}'(\xi)\|^s\dfrac{C^{(1+s)\ell}a_{\min}^{-2\ell s}e^{-s\alpha(B-\sum_{j=1}^{\ell-1}\rho(A_j)-2\ell)}}{\|f_{\x|_1^{\rho(B)}}'(x)\|^{s}\prod_{j=1}^{\ell-1}\|f_{\x|_1^{\rho(A_j)}}'(x)\|^{s}}
\frac{\varepsilon(B)\prod_{j=1}^{\ell-1}\varepsilon(A_j)}{\prod_{j=1}^{\ell}\sum_{k=n_j}^{m_j}\varepsilon(k)}\\
&= \|f_{\ii|_1^k}'(\xi)\|^s\dfrac{C^{(1+s)\ell}a_{\min}^{-2 \ell s}e^{s\alpha(\sum_{j=1}^{\ell-1}(A_j+\rho(A_j))+2\ell)}}{\prod_{j=1}^{\ell}\sum_{k=n_j}^{m_j}\varepsilon(k)} \\
&\leq \delta\cdot\|f_{\ii|_1^k}'(\xi)\|^s,
\end{split}
\]	
where, again, the last inequality follows from Lemma~\ref{lem:existseq}\eqref{it:3} and our choice of $\delta$.

If for every $B\in[n_\ell,R(k)]$ we have $[\ii|_1^k]\cap\Gamma_{A_1,\ldots,A_{\ell-1},B,\ldots}=\emptyset$, then
$$
\eta([\ii|_1^k])=\int_{\Xi}\mu_{\A'}([\ii|_1^k])d\nu(\A')=\sum_{B=R(k)+1}^{m_\ell}\int_{[A_1,\ldots,A_{\ell-1},B]}\mu_\A([\ii|_1^k])d\nu(\A).
$$
First, we give an estimate for the part $B=k,\ldots,m_\ell$. By Lemma~\ref{lem:upperbound} we have
\[
\begin{split}
\sum_{B=k}^{m_\ell}&\int_{[A_1,\ldots,A_{\ell-1},B]}\mu_\A([\ii|_1^k])d\nu(\A)\\
&\ll\sum_{B=k}^{m_\ell}\|f_{\ii|_1^k}'(\xi)\|^s\dfrac{C^{(1+s)\ell}a_{\min}^{-2\ell s}e^{-s\alpha(n_{\ell}-\sum_{j=1}^{\ell-1}\rho(A_j)-2\ell)}}{\prod_{j=1}^{\ell-1}\|f_{\x|_1^{\rho(A_j)}}'(x)\|^{s}}
\frac{\varepsilon(B)\prod_{j=1}^{\ell-1}\varepsilon(A_j)}{\prod_{j=1}^{\ell}\sum_{k=n_j}^{m_j}\varepsilon(k)}\\
&= \sum_{B=k}^{m_\ell}\|f_{\ii|_1^k}'(\xi)\|^s\dfrac{C^{(1+s)\ell}a_{\min}^{-2\ell s}e^{-s\alpha(n_{\ell}-\sum_{j=1}^{\ell-1}(A_j+\rho(A_j))-2\ell)}\varepsilon(B)}{\prod_{j=1}^{\ell}\sum_{k=n_j}^{m_j}\varepsilon(k)}\\
&=\|f_{\ii|_1^k}'(\xi)\|^s\dfrac{C^{(1+s)\ell}a_{\min}^{-2\ell s}e^{-s\alpha(n_{\ell}-\sum_{j=1}^{\ell-1}(A_j+\rho(A_j))-2\ell)}}{\prod_{j=1}^{\ell-1}\sum_{k=n_j}^{m_j}\varepsilon(k)}\frac{\sum_{B=k}^{m_\ell}\varepsilon(B)}{\sum_{k=n_\ell}^{m_\ell}\varepsilon(k)}\\
&\leq\|f_{\ii|_1^k}'(\xi)\|^s\dfrac{C^{(1+s)\ell}a_{\min}^{-2\ell s}e^{-s\alpha(n_{\ell}-\sum_{j=1}^{\ell-1}(A_j+\rho(A_j))-2\ell)}}{\prod_{j=1}^{\ell-1}\sum_{k=n_j}^{m_j}\varepsilon(k)} \\
&\leq \delta \cdot\|f_{\ii|_1^k}'(\xi)\|^s.
\end{split}
\]	
The last inequality above follows from Lemma~\ref{lem:existseq}\eqref{it:4} and our choice of $\delta$.

Finally, to estimate the part corresponding to $B=R(k),\dots,k-1$, note that we have $B < k < B + \rho(B)$. Also note that it can be shown via the chain rule and the bounded distortion property \eqref{eq:boundeddist} that $\|f_{\x|_{1}^{k-B}}(x)\|^{-s} \|f_{\x|_{1}^{\rho(B)}}(x)\|^s = \|f_{\x|_{k-B+1}^{\rho(B)}}(x)\|^s$. Hence, using Lemma~\ref{lem:meas1} directly, we see that
\begin{align*}
\sum_{B=R(k)+1}^{k-1}&\int_{[A_1,\ldots,A_{\ell-1},B]}\mu_\A([\ii|_1^k])d\nu(\A)\\
&\ll\sum_{B=R(k)+1}^{k-1}	\|f_{\ii|_1^k}'(\xi)\|^s\dfrac{C^{(1+s)\ell}a_{\min}^{-2\ell s}e^{-s\alpha(B-\sum_{j=1}^{\ell-1}\rho(A_j)-2\ell)}\cdot\|f_{\x|_1^{k-B}}'(x)\|^{-s}}{\prod_{j=1}^{\ell-1}\|f_{\x|_1^{\rho(A_j)}}'(x)\|^{s}}
\frac{\varepsilon(B)\prod_{j=1}^{\ell-1}\varepsilon(A_j)}{\prod_{j=1}^{\ell}\sum_{k=n_j}^{m_j}\varepsilon(k)}\\
&= \sum_{B=R(k)+1}^{k-1}\|f_{\ii|_1^k}'(\xi)\|^s\dfrac{C^{(1+s)\ell}a_{\min}^{-2\ell s}e^{s\alpha(\sum_{j=1}^{\ell-1}(A_j+\rho(A_j))+2\ell)}\cdot\|f_{\x|_{k-B+1}^{\rho(B)}}'(x)\|^{s}}{\prod_{j=1}^{\ell}\sum_{k=n_j}^{m_j}\varepsilon(k)}.
\end{align*}

Thus, it follows from Lemma~\ref{lem:existseq}\eqref{it:3} and our choice of $\delta$, combined with \eqref{f' composition bounds}, that
\begin{align*}
\sum_{B=R(k)+1}^{k-1}\int_{[A_1,\ldots,A_{\ell-1},B]}\mu_\A([\ii|_1^k])d\nu(\A)
         &\ll\|f_{\ii|_1^k}'(\xi)\|^s\cdot\delta\sum_{B=R(k)+1}^{k-1}(a_{\max})^{s(\rho(B)+B-k)}\\
         &\leq\|f_{\ii|_1^k}'(\xi)\|^s\delta\sum_{B=0}^{\infty}(a_{\max})^{s(\rho(R(k)+B+1)+R(k)+1+B-k)}\\
         &\leq\|f_{\ii|_1^k}'(\xi)\|^s\delta \sum_{B=0}^\infty(a_{\max})^{Bs} \\
         &\ll \delta\|f_{\ii|_i^k}(\xi)\|^s.
\end{align*}

 The proof of Proposition \ref{prop:massdist} is complete upon noting that $\|f_{\ii|_1^k}'(\xi)\|\leq C\diam(X)^{-1}\diam(X_{\ii|_1^k})$ by~\eqref{diameter comparisons}.
\end{proof}

\section{An example: Badly Approximable Numbers and Quadratic Irrationals} \label{Bad and QI application section}

{In this section we discuss an application of our main theorem (Theorem \ref{thm:main}) to the problem of approximating badly approximable numbers by quadratic irrationals. Previously, Baker demonstrated the existence of badly approximable numbers which are ``very well approximable'' by quadratic irrationals \cite{baker}. Here, we extend this result by showing that Theorem \ref{thm:main} can be applied to obtain a Jarn\'{\i}k-type statement for the set of badly approximable numbers which are ``well-approximable'' by a fixed quadratic irrational.} We achieve this by utilising the correspondence between badly approximable numbers and partial quotients of continued fraction expansions, and by expressing the numbers with continued fraction expansions with partial quotients bounded by $M \in \N$ as the attractor of a conformal iterated function system.  

Recall that a number $x\in[0,1]$ is said to be \emph{badly approximable} if there exists a constant $c=c(x)>0$, dependent on $x$, such that for every $\frac{p}{q} \in \Q$ we have
$$
\left|x-\frac{p}{q}\right| > \frac{c}{q^2}.
$$
It is well known that a number is badly approximable if and only if its continued fraction expansion has bounded partial quotients (see, for example, \cite[Theorem 1.15]{BRVaspects}).

For $x \in [0,1]$, let $[a_1,a_2,\dots]$ denote its continued fraction expansion. Note that this expansion will be finite if $x$ is rational. Recall that the $a_i$'s are called the \emph{partial quotients} of $x$ and are the numbers which arise when we write $x$ in the form
\[x=\frac{1}{a_1+\frac{1}{a_2+\frac{1}{a_3+\frac{1}{\phantom{=}\ddots}}}}\]
with $a_i \in \N$ for each $i \in \N$. The partial quotients, $a_i$, can also be defined via the Gauss map. The \emph{Gauss map} is the map $T: [0,1] \to [0,1]$ defined by
\[
T(x)=
\begin{cases}
0 &\text{if } x=0,\\[2ex]
\frac{1}{x} - \left\lfloor \frac{1}{x}\right\rfloor&\text{for } x \in (0,1].
\end{cases}
\]
For each $n \geq 1$, $\displaystyle{a_n=\left\lfloor\frac{1}{T^{n-1}(x)}\right\rfloor}$.

For our present purposes, another useful way for us to view the continued fraction expansion of $x \in [0,1]$ is the following. For every integer $a\geq1$, let
\[f_a(y)=\frac{1}{a+y}.\]
We have
\[ x = [a_1,a_2,\ldots] = \lim_{n\to\infty}f_{a_1}\circ\cdots\circ f_{a_n}(1).\]

If $x\in(0,1]$ is badly approximable, then there exists $Q\geq1$ such that $a_n\leq Q$ for every $n\geq1$. Let us denote the set of badly approximable numbers in $(0,1]$ by $\mathbf{Bad}$, and denote by $\mathbf{Bad}_Q$ the numbers $x\in\mathbf{Bad}$ such that $a_n(x)\leq Q$ for every $n\geq1$ (where $a_n(x)$ is the $n$th partial quotient of $x$). By definition, $\mathbf{Bad}_Q$ is the attractor of the IFS $\{f_a\circ f_b\}_{a,b=1}^Q$. Moreover, it is easy to see that this IFS is conformal and satisfies the open set condition.

Recall that $x \in [0,1]$ is a \emph{quadratic irrational} if it is irrational and is a root of a quadratic polynomial $ax^2+bx+c=0$, where $a\neq0$ and $a,b,c$ are integers. It is well known that $x \in [0,1]$ is a quadratic irrational if and only if $x$ has an \emph{eventually periodic} continued fraction expansion; that is, there exist finite sequences $\ov=(\omega_1,\ldots,\omega_{\ell})$ and $\tv=(\tau_1,\ldots,\tau_m)$ such that
$$
x=[\omega_1,\ldots,\omega_{\ell},\overline{\tau_1,\ldots,\tau_m}],
$$
where $\overline{\tv}$ denotes the infinite periodic sequence formed by repeating $\tv$. We will denote the set of quadratic irrationals in $[0,1]$ by $\mathbf{QI}$.

Given a monotonically decreasing approximating function $\psi\colon\N\mapsto\R^+$, let
$$
W(\psi;\bad;\QI):=\left\{x\in\bad:\exists\alpha\in\QI:|T^n(x)-\alpha|<\psi(n)\text{ for infinitely many }n\in\N\right\}.
$$
Thus, $W(\psi;\bad;\QI)$ is the set of badly approximable numbers which are ``well-approximable'' by a fixed quadratic irrational. We will investigate the Hausdorff measure of $W(\psi;\bad;\QI)$. Notice that it is sensible to fix the quadratic irrational in the definition of $W(\psi;\bad;\QI)$ corresponding to a given $x \in \bad$, otherwise we would necessarily have \mbox{$W(\psi;\bad;\QI) = \bad$} since quadratic irrationals are dense in the reals.

Let us now adapt some standard notation from the usual theory of continued fractions. For the proofs and more details, see \cite[Section~2]{LWWX} or \cite{Khint}.

For a sequence of integers $\{a_n\}_{n=1}^{\infty}$, let
\begin{align*}
q_{n+1}(a_1,\ldots,a_{n+1})&:=q_{n+1}=a_{n+1}q_n+q_{n-1}, \text{ and} \\
\quad p_{n+1}(a_1,\ldots,a_{n+1})&:=p_{n+1}=a_{n+1}p_n+p_{n-1}
\end{align*}
for $n\geq1$, where we define $p_{-1}=q_0=1$ and $p_0=q_{-1}=0$. Then,
$$
f_{a_1}\circ\cdots\circ f_{a_n}(x)=\frac{p_{n-1}x+p_n}{q_{n-1}x+q_n}.
$$
Moreover, for every $0<k<n$ and $a_1,\ldots,a_n\in\N$ we have
\begin{equation}\label{eq:compare}
1\leq\frac{q_n(a_1,\ldots,a_n)}{q_k(a_1,\ldots,a_k)q_{n-k}(a_{k+1},\ldots,a_n)}\leq 2,
\end{equation}
and, for every $x\in[0,1]$,
$$
\frac{1}{4q_n^2}\leq|(f_{a_1}\circ\cdots\circ f_{a_n})'(x)|\leq\frac{1}{q_n^{2}}.
$$
Thus, combining the above bounds with the bounded distortion property \eqref{eq:boundeddist}, there exists a constant $K\geq1$, depending on $Q$ but independent of the sequence $a_1,a_2,\dots,a_n$, such that
\begin{align} \label{example bounded distortion}
\frac{1}{K}\times\frac{1}{q_n(a_1,\ldots,a_n)^2} \leq\diam(f_{a_1}\circ\cdots\circ f_{a_n}(\mathbf{Bad}_Q))\leq \frac{K}{q_n(a_1,\ldots,a_n)^2}.
\end{align}

Now, we are ready to state a corollary of our main result.

\begin{theorem}
	Let $W(\psi; \bad; \QI)$ be the set defined above. We have
	$$
	\mathcal{H}^s(W(\psi;\bad;\QI))=\begin{cases}
	0 & \text{ if,  $\forall \, Q\geq1$, }\sum\limits_{n=1}^\infty\sum\limits_{a_1,\ldots,a_n=1}^Q\dfrac{\psi(n)^s}{q_n(a_1,\ldots,a_n)^{2s}}<\infty, \\&\\
	\infty & \text{ if $\exists \, Q \geq1$ such that }\sum\limits_{n=1}^\infty\sum\limits_{a_1,\ldots,a_n=1}^Q\dfrac{\psi(n)^s}{q_n(a_1,\ldots,a_n)^{2s}}=\infty.
	\end{cases}
	$$
\end{theorem}

\begin{proof}
	If $x\in W(\psi;\bad;\QI)$ then there exists $\alpha=\alpha(x)\in\mathbf{QI}$ such that $|T^n(x)-\alpha|<\psi(n)$ for infinitely many $n\in\N$. On the other hand, there exists $Q_1\geq1$ such that $x\in\mathbf{Bad}_{Q_1}$ and since $\alpha$ is quasi-periodic there exists $Q_2\geq1$ such that $\alpha\in\mathbf{Bad}_{Q_2}$. Hence, $x\in\mathbf{Bad}_{Q}$ and $\alpha\in\mathbf{QI}\cap\mathbf{Bad}_Q$, where $Q=\max\{Q_1,Q_2\}$. Thus, $x\in W_Q(\psi;\alpha)$
	where
	\[W_Q(\psi;\alpha)=\left\{x\in\mathbf{Bad}_Q:|T^{n}(x)-\alpha|<\psi(n)\text{ for infinitely many }n \in \N\right\}.\]
	So, we have that
	\begin{equation}\label{eq:decompbad}
	W(\psi;\bad;\QI)=\bigcup_{Q=1}^{\infty}\bigcup_{\alpha\in\mathbf{QI}\cap\mathbf{Bad}_Q}W_Q(\psi;\alpha).
	\end{equation}

	It is easy to see that if $x\in W_Q(\psi;\alpha)$ then we must have $|T^{2n}(x)-\alpha|<\psi(2n)$ for infinitely many $n\in\N$ or $|T^{2n+1}(x)-\alpha|<\psi(2n+1)$ for infinitely many $n\in\N$. Hence, we can decompose $W_{Q}(\psi;\alpha)$ into
	\[W_Q(\psi;\alpha)=U_{Q}(\psi_0;\alpha)\cup T(U_{Q}(\psi_1;\alpha)),\] where $\psi_0(n)=\psi(2n)$ and $\psi_1(n)=\psi(2n+1)$ and
	$$
	U_{Q}(\psi_i;\alpha)=\{x\in\mathbf{Bad}_Q:|T^{2n}(x)-\alpha|<\psi_i(n)\text{ for infinitely many } n \in \N\}.
	$$
	
	Since the set $\mathbf{Bad}_Q$ is the attractor of the conformal IFS $\{f_a\circ f_b\}_{a,b=1}^Q$, it follows from Theorem~\ref{thm:main} taken together with \eqref{example bounded distortion} that, for any $\alpha\in\mathbf{QI}\cap\mathbf{Bad}_Q$, we have
	\begin{align}\label{eq:cor}
	\mathcal{H}^s(U_{Q}(\psi_i, \alpha))=\begin{cases}
	0 & \text{ if }\sum_{n=1}^\infty\sum_{a_1=1,\ldots,a_{2n}=1}^Q\frac{\psi_i(n)^s}{q_{2n}(a_1,\ldots,a_{2n})^{2s}}<\infty \\&\\
	\infty & \text{ if }\sum_{n=1}^\infty\sum_{a_1=1,\ldots,a_{2n}=1}^Q\frac{\psi_i(n)^s}{q_{2n}(a_1,\ldots,a_{2n})^{2s}}=\infty.
	\end{cases}
	\end{align}
	On the other hand,
	$$
	\sum\limits_{n=1}^\infty\sum\limits_{a_1,\ldots,a_n=1}^Q\dfrac{\psi(n)^s}{q_n(a_1,\ldots,a_n)^{2s}}=\sum\limits_{n=1}^\infty\sum\limits_{a_1,\ldots,a_{2n}=1}^Q\dfrac{\psi(2n)^s}{q_{2n}(a_1,\ldots,a_{2n})^{2s}}+\sum\limits_{n=0}^\infty\sum\limits_{a_1,\ldots,a_{2n+1}=1}^Q\dfrac{\psi(2n+1)^s}{q_{2n+1}(a_1,\ldots,a_{2n+1})^{2s}}.
	$$
	By \eqref{eq:compare},
	\begin{align*}
	\sum\limits_{n=0}^\infty\sum\limits_{a_1,\ldots,a_{2n+1}=1}^Q\dfrac{\psi(2n+1)^s}{q_{2n+1}(a_1,\ldots,a_{2n+1})^{2s}}&\leq2\sum\limits_{n=0}^\infty\sum\limits_{a_1,\ldots,a_{2n+1}=1}^Q\dfrac{\psi(2n+1)^s}{q_{1}(a_{2n+1})^{2s}q_{2n}(a_1,\ldots,a_{2n})^{2s}}\\
	&=2\left(\sum\limits_{a=1}^Qq_{1}(a)^{-2s}\right)\sum\limits_{n=0}^\infty\sum\limits_{a_1,\ldots,a_{2n}=1}^Q\dfrac{\psi(2n+1)^s}{q_{2n}(a_1,\ldots,a_{2n})^{2s}},\\
	\end{align*}
	and the other inequality $$\sum\limits_{n=0}^\infty\sum\limits_{a_1,\ldots,a_{2n+1}=1}^Q\dfrac{\psi(2n+1)^s}{q_{2n+1}(a_1,\ldots,a_{2n+1})^{2s}}\geq\left(\sum\limits_{a=1}^Qq_{1}(a)^{-2s}\right)\sum\limits_{n=0}^\infty\sum\limits_{a_1,\ldots,a_{2n}=1}^Q\dfrac{\psi(2n+1)^s}{q_{2n}(a_1,\ldots,a_{2n})^{2s}}$$ follows by similar argument. Thus, $\sum\limits_{n=1}^\infty\sum\limits_{a_1,\ldots,a_n=1}^Q\dfrac{\psi(n)^s}{q_n(a_1,\ldots,a_n)^{2s}}$ is finite if and only if $$\sum_{n=1}^\infty\sum_{a_1=1,\ldots,a_{2n}=1}^Q\frac{\psi_i(n)^s}{q_{2n}(a_1,\ldots,a_{2n})^{2s}}<\infty$$ for $i=0,1$. Then the statement follows by \eqref{eq:decompbad} and \eqref{eq:cor}.
\end{proof}

\vspace{0.5cm}
\noindent {\bf Acknowledgements.}
This project grew out of initial discussions had while both authors were in attendance at the program on \emph{Fractal Geometry and Dynamics} at the Mittag--Leffler Institut in November 2017. We are indebted both to the organisers of the program and the staff at the Institut for a pleasant and productive stay at the Institut. The first author is grateful to the Budapest University of Technology and Economics for their hospitality during her visit there in October 2018. She would also like to thank Tom Kempton and Charles Walkden for patiently listening to her ramblings on this project. Both authors would like to thank Thomas Jordan for pointing out several useful references. We also thank the anonymous referee for a number of useful comments.

\bibliographystyle{abbrv}

\end{document}